\documentclass[11pt]{article}

\usepackage[a4paper,margin=1in]{geometry}
\usepackage[T1]{fontenc}
\usepackage[utf8]{inputenc}
\usepackage{lmodern}
\usepackage{amsmath,amssymb,amsthm,mathtools,bm,mathrsfs}
\usepackage{booktabs,array,longtable}
\usepackage{float}
\usepackage{graphicx}
\usepackage{tikz}
\usetikzlibrary{calc,positioning,arrows.meta}
\usepackage{xcolor}
\usepackage{enumitem}
\usepackage{microtype}
\usepackage[unicode]{hyperref}
\usepackage{cleveref}
\usepackage{authblk}

\hypersetup{
  colorlinks=true,
  linkcolor=blue!55!black,
  citecolor=blue!55!black,
  urlcolor=blue!55!black,
  pdftitle={PH2T-splines, Part I: A Reasonable Mesh Assumption},
  pdfauthor={Bingru Huang and Yue Xi},
  pdfsubject={Dimensional stability of highest-smoothness splines on hierarchical T-meshes},
  pdfkeywords={hierarchical T-mesh, polynomial spline, dimensional stability, template refinement}
}

\newtheorem{theorem}{Theorem}[section]
\newtheorem{proposition}[theorem]{Proposition}

\newtheorem{corollary}[theorem]{Corollary}

\newtheorem{assumption}[theorem]{Assumption}
\theoremstyle{definition}
\newtheorem{definition}[theorem]{Definition}
\newtheorem{example}[theorem]{Example}
\theoremstyle{remark}

\crefname{theorem}{Theorem}{Theorems}
\Crefname{theorem}{Theorem}{Theorems}
\crefname{proposition}{Proposition}{Propositions}
\Crefname{proposition}{Proposition}{Propositions}
\crefname{lemma}{Lemma}{Lemmas}
\Crefname{lemma}{Lemma}{Lemmas}
\crefname{corollary}{Corollary}{Corollaries}
\Crefname{corollary}{Corollary}{Corollaries}
\crefname{conjecture}{Conjecture}{Conjectures}
\Crefname{conjecture}{Conjecture}{Conjectures}
\crefname{assumption}{Assumption}{Assumptions}
\Crefname{assumption}{Assumption}{Assumptions}
\crefname{definition}{Definition}{Definitions}
\Crefname{definition}{Definition}{Definitions}
\crefname{remark}{Remark}{Remarks}
\Crefname{remark}{Remark}{Remarks}
\crefname{example}{Example}{Examples}
\Crefname{example}{Example}{Examples}

\newcommand{\T}{\mathscr T}
\newcommand{\Q}{\mathbb Q}
\newcommand{\R}{\mathbb R}
\newcommand{\rank}{\operatorname{rank}}
\newcommand{\chl}{\operatorname{chl}}
\newcommand{\cR}{\mathcal R}
\newcommand{\cQ}{\mathcal Q}
\newcommand{\PHtwoT}{\mathrm{PH}^{2}\mathrm{T}}

\setlist{itemsep=2pt,topsep=4pt}
\title{\Large\bfseries $\PHtwoT$-splines, Part I:\\[3pt]
A Reasonable Mesh Assumption}
\author[1]{Bingru Huang\thanks{Corresponding author. E-mail: \href{mailto:hbr999@ustc.edu.cn}{\texttt{hbr999@ustc.edu.cn}}.}}
\author[1]{Yue Xi}
\affil[1]{School of Mathematical Sciences, University of Science and Technology of China, Hefei 230026, P.R. China}
\date{}

\begin{document}
\maketitle

\begin{abstract}
This paper is the first in a three-part series on the construction of
polynomial splines with the highest order of smoothness over hierarchical
T-meshes, referred to as $\PHtwoT$-splines. For splines of bi-degree
$(d,d)$, we study suitable refinement conditions for the subsequent
basis construction, which requires dimensional stability of the
underlying spline space. We present two groups of examples, considering
unrestricted hierarchical refinement and refinement without vanishable
T $l$-edges, respectively. The first setting permits new edges without
additional degrees of freedom. In the second group, every refinement
level excludes vanishable T $l$-edges and increases the dimension.
Nevertheless, the dimension is unstable in both groups. We then
introduce template translations to describe each refined region as a
union of translates of a fixed template in the cell-index grid.
Together with the known stability result under
$(d-1)\times(d-1)$ template refinement, these examples justify this
condition as a reasonable mesh assumption for the subsequent
$\PHtwoT$-spline construction.
\end{abstract}


\section{Introduction}

Non-uniform rational B-splines (NURBS) are widely used in computer-aided
geometric design and industrial CAD systems. They provide a unified
representation of curves and surfaces, exact representations of conic
sections, affine invariance, and efficient evaluation algorithms
\cite{PieglTiller1997}. In isogeometric analysis, spline functions
are used both to represent the geometry and to construct the
approximation spaces, thereby reducing the gap between geometric
modeling and numerical analysis
\cite{HughesCottrellBazilevs2005,CottrellHughesBazilevs2009}.

The tensor-product structure of NURBS, however, limits local
refinement. Inserting a knot in one parameter direction introduces a
grid line across the entire tensor-product patch, which may add many
degrees of freedom outside the region to be refined. Several locally
refinable spline constructions have been developed to overcome this
limitation. These include T-splines and their local refinement
algorithms \cite{SederbergEtAl2003,SederbergEtAl2004},
analysis-suitable and analysis-suitable++ T-splines
\cite{ScottEtAl2012,ZhangLi2018}, hierarchical and truncated
hierarchical B-splines
\cite{ForseyBartels1988,VuongEtAl2011,GiannelliJuettlerSpeleers2012},
locally refined splines \cite{DokkenLychePettersen2013}, and polynomial
splines over hierarchical T-meshes \cite{DengEtAl2008}. A review of
these constructions can be found in \cite{LiEtAl2016Survey}.

Hierarchical B-splines (HB-splines) are constructed by selecting
tensor-product B-splines level by level from nested spline spaces
according to prescribed refinement regions. The resulting functions
are linearly independent and retain the computational advantages of
standard B-splines \cite{ForseyBartels1988,VuongEtAl2011}. Their
completeness in the full piecewise polynomial space on the hierarchical
mesh requires further analysis. Dimension formulas on multi-cell and
multi-grid domains, together with sufficient conditions for HB
completeness based on successive refinement rings, were established
in \cite{GiannelliJuettler2013}. Truncated hierarchical B-splines
(THB-splines) form a partition of unity and reduce the supports of
coarse-level functions while preserving the span of the HB-splines
\cite{GiannelliJuettlerSpeleers2012}. Thus, truncation does not change
whether the hierarchical space coincides with the full spline space.

Polynomial splines over T-meshes are defined directly by prescribing
the polynomial bi-degree on each cell and the smoothness across
interior edges \cite{DengChenFeng2006,DengEtAl2008}. This definition
does not depend on a particular generating system. Two fundamental
problems are to calculate the dimension of the spline space and to
construct a basis. For bicubic splines with $C^{1,1}$ continuity,
the PHT-spline construction provides locally refinable bases over
hierarchical T-meshes \cite{DengEtAl2008}. In this paper, we are
concerned with the spline space $S_d(\T)$ of bi-degree $(d,d)$ with
the highest order of smoothness $C^{d-1,d-1}$ over a hierarchical
T-mesh $\T$.

Several methods have been developed to calculate the dimensions of
spline spaces over T-meshes. Early dimension formulas were obtained
by the B-net method under restrictions on the degree and smoothness
\cite{DengChenFeng2006}, and were further improved by the smoothing
cofactor method \cite{LiWangZhang2006}. The latter relates the
dimension calculation to a system of conformality equations
associated with the mesh \cite{LiDeng2016}. The homology method
provides another dimension formula, which contains a term that
cannot in general be determined from the topological information
alone \cite{Mourrain2014}.

When the smoothness order is close to the polynomial degree, the
dimension of the spline space may not be stable; that is, it may
depend not only on the topological information but also on the
geometric information of the T-mesh. This phenomenon was first
observed in \cite{LiChen2011}, and further examples were provided
in \cite{BerdinskyEtAl2012}. The instability of the dimension over
T-meshes containing T-cycles and nested T-cycles was subsequently
studied in \cite{GuoWangLi2015,LiWang2019}.

These observations suggest studying restricted classes of T-meshes
over which the dimensions of spline spaces are stable. For
hierarchical T-meshes, a dimension formula for biquadratic spline
spaces was obtained by the space embedding method
\cite{DengChenJin2013}. The highest-smoothness case over
$(m,n)$-subdivision T-meshes was studied in \cite{WuDengChen2013},
and biquadratic and bicubic spaces over hierarchical T-meshes with
fewer refinement restrictions were considered in \cite{ZengEtAl2015}.
Another important class is diagonalizable T-meshes, whose T $l$-edges
admit a suitable ordering. The dimension of a spline space over such
a mesh is determined by its topological information
\cite{LiDeng2016}.

For the spline space $S_d(\T)$, a decomposition of the T-connected
component into a diagonalizable component and a completely
non-diagonalizable component was developed in \cite{HuangChen2024}.
It was proved that the stability of the dimension depends only on
the stability of the rank of a reduced matrix corresponding to the
multi-vertices of the completely non-diagonalizable component.
A subsequent study calculated the dimensions of conformality vector
spaces over tensor-product T-connected components and established
a levelwise dimension formula under suitable hierarchical
subdivisions \cite{HuangChen2026}. In particular, for $d\ge3$, the dimension of
$S_d(\T)$ is stable when each refinement level cross-subdivides a
collection of $(d-1)\times(d-1)$ tensor-product cell blocks, with
partial overlaps permitted. This result provides the
dimension-theoretic foundation for the present work.

This paper is the first in a three-part series on the construction
of $\PHtwoT$-splines. We study suitable refinement
conditions for constructing a basis of the full spline space
$S_d(\T)$ over hierarchical T-meshes. A necessary requirement for
the proposed mesh-based construction is the stability of the
dimension. We examine hierarchical refinements with and without
vanishable T $l$-edges through four examples, including the biquartic
example in \cite{HuangChen2026}. The dimension calculations show
that neither refinement condition is sufficient for stability.
Together with the known dimension formula, these results motivate
the $(d-1)\times(d-1)$ template assumption for the subsequent
$\PHtwoT$-spline construction. The main
contributions are as follows.
\begin{itemize}
    \item
    The notion of construction-suitable T-meshes is introduced, and
    dimensional stability is proved to be a necessary condition
    for a basis construction compatible with the mesh structure.

    \item
    Two groups of examples show that neither unrestricted
    hierarchical refinement for $d=3,4$ nor refinement without
    vanishable T $l$-edges for $d=5,6$ guarantees dimensional
    stability. In the latter group, each refinement level strictly
    increases the dimension.

    \item
    Template translations are introduced to describe the refined
    regions, and $\lceil(d+1)/2\rceil$ is shown to be the smallest
    square-template side that uniformly excludes vanishable T
    $l$-edges. For $d=5,6$, the resulting
    $(d-2)\times(d-2)$ template condition is nevertheless
    insufficient for dimensional stability.
\end{itemize}

The rest of the paper is organized as follows.
Section~\ref{sec:prelim} recalls the basic notions of T-meshes and
spline spaces. Section~\ref{sec:construction-suitable} introduces
construction-suitable T-meshes and establishes the necessary condition
of dimensional stability. Section~\ref{sec:instability-assumption}
presents the four instability examples and develops the resulting
mesh assumption. Section~\ref{sec:conclusion} concludes the paper.
The appendices provide the conformality matrix reductions, the proofs
of the dimension results, and the proofs of the template conditions.

\section{Preliminaries}\label{sec:prelim}

In this section, we recall some basic notions of T-meshes, hierarchical
T-meshes and spline spaces. We also review the smoothing cofactor
equations used below and the definition of dimensional stability.

\subsection{T-meshes and hierarchical T-meshes}

\begin{definition}[T-mesh; \cite{DengEtAl2008,HuangChen2024}]
\label{def:tmesh}
A T-mesh $\T$ is a finite collection of closed axis-aligned rectangles with
pairwise disjoint interiors and connected union. Each rectangle is
called a \emph{cell}, and their union is denoted by $\Omega(\T)$.
If $\Omega(\T)$ is a rectangle, then $\T$ is called a
\emph{regular T-mesh}.
\end{definition}

Throughout this paper, we consider regular T-meshes. Vertices and
edges are classified as boundary or interior according to their
positions. An interior vertex at which three edges meet is called a
\emph{T-node}, whereas one at which four edges meet is called a
\emph{cross-vertex}.

An \emph{$l$-edge} is a maximal line segment composed of collinear
mesh edges. An $l$-edge on the boundary of $\Omega(\T)$ is called a
\emph{boundary $l$-edge}. An interior $l$-edge is a \emph{cross-cut}
if both endpoints are boundary vertices, a \emph{ray} if exactly one
endpoint is a boundary vertex, and a \emph{T $l$-edge} if both
endpoints are T-nodes. The union of all T $l$-edges and their vertices
is called the \emph{T-connected component}, denoted by $L(\T)$
\cite{LiDeng2016,HuangChen2024}. If it is not connected, its connected
components can be considered separately. A vertex on two T $l$-edges
is called a \emph{multi-vertex}; any other vertex on $L(\T)$ is called
a \emph{mono-vertex}. We denote the number of vertices on an $l$-edge
$\ell$ by $n(\ell)$.

\begin{example}\label{ex:tedges}
In Figure~\ref{fig:tedges}, $c$ is a cross-cut, $r$ is a ray, and
$\ell_h$ and $\ell_v$ are T $l$-edges. The vertex $v$ is a
multi-vertex, whereas $w$ is a mono-vertex.
\end{example}

\begin{figure}[htbp]
\centering
\begin{tikzpicture}[scale=.83,line cap=round,line join=round]
  \draw[line width=.6pt] (0,0) rectangle (7,5);
  \draw[line width=.6pt] (0,2.7)--(7,2.7);
  \draw[line width=.6pt] (0,4.3)--(7,4.3);
  \draw[line width=.6pt] (2,0)--(2,5);
  \draw[line width=.6pt] (5,0)--(5,5);

  \draw[line width=1.05pt] (0,1)--(7,1);
  \draw[line width=1.05pt] (4.5,0)--(4.5,2.7);
  \draw[line width=1.2pt] (2,3.5)--(5,3.5);
  \draw[line width=1.2pt] (3.5,2.7)--(3.5,4.3);

  \fill (2,3.5) circle (1.8pt);
  \fill (5,3.5) circle (1.8pt);
  \fill (3.5,2.7) circle (1.8pt);
  \fill (3.5,4.3) circle (1.8pt);
  \fill (3.5,3.5) circle (2.5pt);

  \node[fill=white,inner sep=1pt] at (1.05,.78) {$c$};
  \node[fill=white,inner sep=1pt] at (4.83,1.72) {$r$};
  \node[fill=white,inner sep=1pt] at (4.72,3.76) {$\ell_h$};
  \node[fill=white,inner sep=1pt] at (3.78,4.05) {$\ell_v$};
  \node[fill=white,inner sep=1pt] at (3.72,3.28) {$v$};
  \node[fill=white,inner sep=1pt] at (2.22,3.78) {$w$};
\end{tikzpicture}
\caption{A T-mesh.}
\label{fig:tedges}
\end{figure}
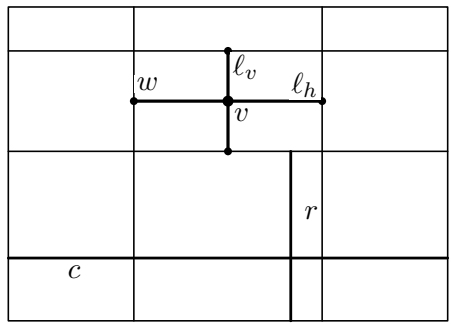

A \emph{hierarchical T-mesh} is generated from a tensor-product mesh
by successive local subdivisions \cite{DengEtAl2008,HuangChen2026}.
The initial tensor-product mesh is of level zero. At each subsequent
level, some cells of level $k-1$ are subdivided, and the newly generated
cells, edges and vertices are assigned level $k$. The cells not
selected for subdivision remain unchanged. In this paper, we use the
\emph{cross subdivision mode}, in which a cell is divided into
$2\times2$ equal subcells by its horizontal and vertical midlines.

Let $\T^k$ denote the T-mesh after the subdivisions up to level $k$.
The resulting sequence is written as
\[
  \T^0,\T^1,\ldots,\T^N=\T,
\]
where $N=\operatorname{lev}(\T)$ is the level of the hierarchical
T-mesh. The collection of cells subdivided from $\T^{k-1}$ to
$\T^k$ determines the refined region at level $k$.

\begin{example}\label{ex:hierarchical}
Figure~\ref{fig:hierarchical} shows a hierarchical T-mesh of level two.
The central cell $[1,2]^2$ of $\T^0$ is subdivided to obtain $\T^1$,
and the cell $[3/2,2]^2$ is then subdivided to obtain $\T^2$.
The new line segments of levels one and two are shown as solid and
dashed lines, respectively.
\end{example}

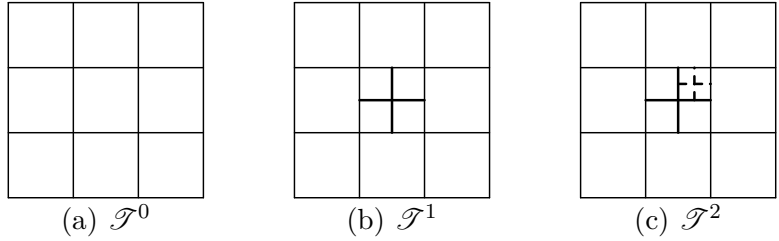
\begin{figure}[htbp]
\centering
\begin{tikzpicture}[scale=.86,line cap=round,line join=round]
  \begin{scope}
    \foreach \x in {0,1,2,3}
      {\draw[line width=.55pt] (\x,0)--(\x,3);}
    \foreach \y in {0,1,2,3}
      {\draw[line width=.55pt] (0,\y)--(3,\y);}
    \node at (1.5,-.35) {(a) $\T^0$};
  \end{scope}

  \begin{scope}[shift={(4.4,0)}]
    \foreach \x in {0,1,2,3}
      {\draw[line width=.55pt] (\x,0)--(\x,3);}
    \foreach \y in {0,1,2,3}
      {\draw[line width=.55pt] (0,\y)--(3,\y);}
    \draw[line width=.95pt] (1.5,1)--(1.5,2);
    \draw[line width=.95pt] (1,1.5)--(2,1.5);
    \node at (1.5,-.35) {(b) $\T^1$};
  \end{scope}

  \begin{scope}[shift={(8.8,0)}]
    \foreach \x in {0,1,2,3}
      {\draw[line width=.55pt] (\x,0)--(\x,3);}
    \foreach \y in {0,1,2,3}
      {\draw[line width=.55pt] (0,\y)--(3,\y);}
    \draw[line width=.95pt] (1.5,1)--(1.5,2);
    \draw[line width=.95pt] (1,1.5)--(2,1.5);
    \draw[line width=.95pt,dashed] (1.75,1.5)--(1.75,2);
    \draw[line width=.95pt,dashed] (1.5,1.75)--(2,1.75);
    \node at (1.5,-.35) {(c) $\T^2$};
  \end{scope}
\end{tikzpicture}
\caption{A hierarchical T-mesh.}
\label{fig:hierarchical}
\end{figure}

\subsection{Spline spaces and smoothing cofactors}

\begin{definition}[Spline space; \cite{DengEtAl2008,HuangChen2024}]
\label{def:spline-space}
Let $\T$ be a regular T-mesh with domain $\Omega$, and let $d\ge1$.
The spline space of bi-degree $(d,d)$ with the highest order of
smoothness over $\T$ is defined by
\begin{equation}\label{eq:spline-space}
  S_d(\T)
  =\left\{
    f\in C^{d-1,d-1}(\Omega):
    f|_Q\in\mathbb P_{d,d}\text{ for every }Q\in\T
  \right\},
\end{equation}
where $\mathbb P_{d,d}$ is the space of polynomials of degree at most
$d$ in each variable, and $C^{d-1,d-1}(\Omega)$ denotes
$C^{d-1}$ continuity in both coordinate directions.
\end{definition}

No boundary conditions are imposed on $S_d(\T)$. Thus,
\eqref{eq:spline-space} defines the full piecewise polynomial space,
not the span of a prescribed family of splines.

\begin{example}\label{ex:spline-space}
For an $n_x\times n_y$ tensor-product mesh with simple interior knots,
\[
  \dim S_d(\T)=(n_x+d)(n_y+d).
\]
In particular, the mesh in Figure~\ref{fig:hierarchical}(a) has
$\dim S_1(\T^0)=16$.
\end{example}

The smoothing cofactor method expresses the smoothness conditions
in terms of edge and vertex cofactors
\cite{LiDeng2016,HuangChen2024}. Consider a horizontal T $l$-edge
$\ell$ with $r$ vertices whose $x$-coordinates are
$s_1<s_2<\cdots<s_r$. Let $\delta_1,\delta_2,\ldots,\delta_r$ be
the corresponding vertex cofactors; see Figure~\ref{fig:cofactor-edge}.
The conformality condition along $\ell$ is
\begin{equation}\label{eq:global-conformality}
  \sum_{i=1}^{r}\delta_i(x-s_i)^d=0.
\end{equation}
Comparing the coefficients of the powers of $x$ gives
\begin{equation}\label{eq:moment}
  \sum_{i=1}^{r}\delta_i s_i^k=0,
  \qquad k=0,1,\ldots,d.
\end{equation}
Equivalently,
\begin{equation}\label{eq:edge-conformality-matrix}
  V_\ell
  \begin{pmatrix}
    \delta_1\\
    \delta_2\\
    \vdots\\
    \delta_r
  \end{pmatrix}
  =\mathbf 0,
  \qquad
  V_\ell=
  \begin{pmatrix}
    1      & 1      & \cdots & 1\\
    s_1    & s_2    & \cdots & s_r\\
    \vdots & \vdots &        & \vdots\\
    s_1^d  & s_2^d  & \cdots & s_r^d
  \end{pmatrix}.
\end{equation}
For a vertical T $l$-edge, the same equations hold with $s_i$ replaced
by the corresponding $y$-coordinates $t_i$. The systems associated
with all T $l$-edges form the \emph{global conformality condition}
of $S_d(\T)$.

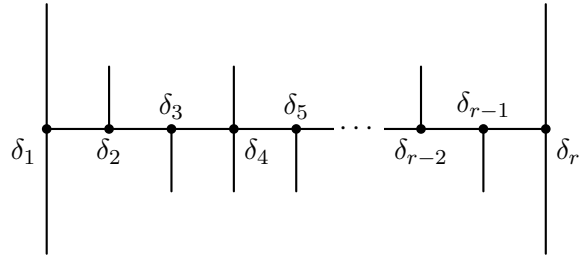
\begin{figure}[htbp]
\centering
\begin{tikzpicture}[scale=1.65,line cap=round,line join=round]
  \draw[line width=.8pt] (1,.5)--(1,2.5);
  \draw[line width=.8pt] (1,1.5)--(5,1.5);
  \draw[line width=.8pt] (5,.5)--(5,2.5);
  \draw[line width=.8pt] (1.5,1.5)--(1.5,2);
  \draw[line width=.8pt] (2,1)--(2,1.5);
  \draw[line width=.8pt] (2.5,1)--(2.5,2);
  \draw[line width=.8pt] (3,1)--(3,1.5);
  \draw[line width=.8pt] (4,1.5)--(4,2);
  \draw[line width=.8pt] (4.5,1)--(4.5,1.5);

  \foreach \x in {1,1.5,2,2.5,3,4,4.5,5}
    {\fill (\x,1.5) circle (1.1pt);}

  \node[below left]  at (1,1.5) {$\delta_1$};
  \node[below]       at (1.5,1.5) {$\delta_2$};
  \node[above]       at (2,1.5) {$\delta_3$};
  \node[below right] at (2.5,1.5) {$\delta_4$};
  \node[above]       at (3,1.5) {$\delta_5$};
  \node[below]       at (4,1.5) {$\delta_{r-2}$};
  \node[above]       at (4.5,1.5) {$\delta_{r-1}$};
  \node[below right] at (5,1.5) {$\delta_r$};
  \node[fill=white,inner sep=2pt] at (3.5,1.5) {$\cdots$};
\end{tikzpicture}
\caption{Vertex cofactors along a horizontal T $l$-edge.}
\label{fig:cofactor-edge}
\end{figure}

Since $s_1,\ldots,s_r$ are distinct,
$\operatorname{rank}V_\ell=\min\{d+1,r\}$. Consequently,
\begin{equation}\label{eq:localnullity}
  \dim\ker V_\ell=\max\{0,r-d-1\}.
\end{equation}
In particular, when $r\le d+1$, all vertex cofactors on $\ell$
are zero.

\begin{definition}[Vanishable T $l$-edge]\label{def:vanishable}
A T $l$-edge $\ell$ is called \emph{vanishable} for $S_d$ if its
conformality equations have only the zero solution, or equivalently,
if $n(\ell)\le d+1$. Otherwise, it is called \emph{non-vanishable}.
A hierarchical refinement is called \emph{no-vanishable} if every
T $l$-edge newly generated at each level is non-vanishable immediately
after that subdivision.
\end{definition}

\begin{example}\label{ex:vanishable-levels}
Let $d=5$. Cross subdivision of an isolated $2\times2$ submesh produces
T $l$-edges with five vertices, which are vanishable. For an isolated
$3\times3$ submesh, each new T $l$-edge has seven vertices and is
non-vanishable; see Figure~\ref{fig:vanishable-levels}. This distinction
concerns the conformality equations on individual T $l$-edges and
does not determine the dimension of the full spline space.
\end{example}

\begin{figure}[htbp]
\centering
\begin{tikzpicture}[scale=.63,line cap=round,line join=round]
  \begin{scope}
    \foreach \x in {0,...,4}
      {\draw[gray!55] (\x,0)--(\x,4);}
    \foreach \y in {0,...,4}
      {\draw[gray!55] (0,\y)--(4,\y);}

    \foreach \i in {1,2}{\foreach \j in {1,2}{
      \draw[blue!70!black,line width=.65pt]
        (\i+.5,\j)--(\i+.5,\j+1);
      \draw[blue!70!black,line width=.65pt]
        (\i,\j+.5)--(\i+1,\j+.5);
    }}

    \draw[red!75!black,line width=1.15pt] (1,1.5)--(3,1.5);
    \foreach \x in {1,1.5,2,2.5,3}
      {\fill[red!75!black] (\x,1.5) circle (2pt);}
    \node at (2,-.38) {(a) $2\times2$ submesh};
  \end{scope}

  \begin{scope}[shift={(6,-.5)}]
    \foreach \x in {0,...,5}
      {\draw[gray!55] (\x,0)--(\x,5);}
    \foreach \y in {0,...,5}
      {\draw[gray!55] (0,\y)--(5,\y);}

    \foreach \i in {1,2,3}{\foreach \j in {1,2,3}{
      \draw[blue!70!black,line width=.65pt]
        (\i+.5,\j)--(\i+.5,\j+1);
      \draw[blue!70!black,line width=.65pt]
        (\i,\j+.5)--(\i+1,\j+.5);
    }}

    \draw[red!75!black,line width=1.15pt] (1,1.5)--(4,1.5);
    \foreach \x in {1,1.5,2,2.5,3,3.5,4}
      {\fill[red!75!black] (\x,1.5) circle (2pt);}
    \node at (2.5,-.38) {(b) $3\times3$ submesh};
  \end{scope}
\end{tikzpicture}
\caption{Two refinements for $d=5$.}
\label{fig:vanishable-levels}
\end{figure}
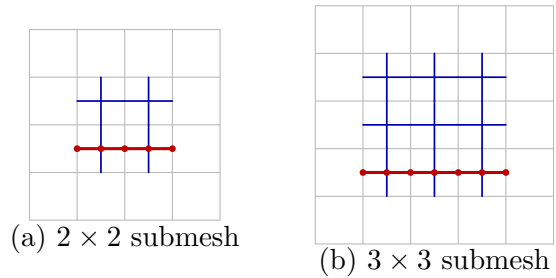

\subsection{Structurally isomorphic T-meshes and dimensional stability}

Let $\mathcal E(\T)$ contain all $l$-edges of $\T$, including boundary
$l$-edges. Its horizontal and vertical subsets are denoted by
$\mathcal E_1(\T)$ and $\mathcal E_2(\T)$, respectively. Write
$B(\T)$, $C(\T)$, $R(\T)$ and $\mathcal E_{\mathrm T}(\T)$ for the
sets of boundary $l$-edges, cross-cuts, rays and T $l$-edges.
For an $l$-edge $\ell$, $\operatorname{cor}(\ell)$ denotes its constant
$y$-coordinate if it is horizontal, and its constant $x$-coordinate
if it is vertical.

\begin{definition}[Structurally isomorphic map;
{\cite[Definition~3.1]{HuangChen2025Classification}}]
\label{def:structural-isomorphism}
Two T-meshes $\T_1$ and $\T_2$ are \emph{structurally isomorphic}
if a bijection
\[
  \varphi:\mathcal E(\T_1)\longrightarrow\mathcal E(\T_2)
\]
exists with the following properties:
\begin{enumerate}[label=\textup{(\arabic*)},leftmargin=*]
  \item
  The four types of $l$-edges are preserved:
  \[
  \begin{aligned}
    \varphi(B(\T_1))&=B(\T_2), &
    \varphi(C(\T_1))&=C(\T_2),\\
    \varphi(R(\T_1))&=R(\T_2), &
    \varphi(\mathcal E_{\mathrm T}(\T_1))
      &=\mathcal E_{\mathrm T}(\T_2).
  \end{aligned}
  \]

  \item
  Intersections are preserved: for all
  $\ell_i,\ell_j\in\mathcal E(\T_1)$,
  \[
    \ell_i\cap\ell_j\ne\varnothing
    \quad\Longleftrightarrow\quad
    \varphi(\ell_i)\cap\varphi(\ell_j)\ne\varnothing.
  \]

  \item
  For $k=1,2$ and $\ell_i,\ell_j\in\mathcal E_k(\T_1)$, the mutual
  positions of parallel $l$-edges satisfy
  \[
    \operatorname{cor}(\ell_i)<\operatorname{cor}(\ell_j)
    \quad\Longrightarrow\quad
    \operatorname{cor}(\varphi(\ell_i))
      <\operatorname{cor}(\varphi(\ell_j)).
  \]
\end{enumerate}
The bijection $\varphi$ is called a \emph{structurally isomorphic map}.
\end{definition}

Write $\T_1\sim_{\mathrm I}\T_2$ when such a map exists. For a
specified collection $U$ of T-meshes, the \emph{structurally isomorphic
class} of $\T\in U$ is
\[
  [\T]_{\mathrm I}
  =\{\widehat\T\in U:\widehat\T\sim_{\mathrm I}\T\}.
\]
For hierarchical T-meshes, we work in the admissible collection $U$
with prescribed levelwise cross subdivisions. The initial grid-line
coordinates may vary, while the selected cell indices and midpoint
rule remain fixed. These variations preserve the order and coincidences
of the supporting grid lines. The corresponding subdivisions are part
of the mesh data used below. Dimensional stability is always understood
relative to this specified collection $U$.

\begin{example}\label{ex:structural-isomorphism}
Consider two tensor-product meshes with grid-line coordinates
\[
\begin{aligned}
  X_{\mathrm s}&=(0,1,2,3),\qquad Y_{\mathrm s}=(0,1,2,3),\\
  X_{\mathrm l}&=(0,2,5,9),\qquad
  Y_{\mathrm l}=(0,1,4,7).
\end{aligned}
\]
Cross-subdivide their central cells
$R_{\mathrm s}=[1,2]^2$ and $R_{\mathrm l}=[2,5]\times[1,4]$,
respectively. The resulting meshes $\T_{\mathrm s}$ and $\T_{\mathrm l}$
are shown in Figure~\ref{fig:hb-isomorphic-meshes}. Mapping each
$l$-edge to its counterpart preserves its type, its intersections
with other $l$-edges and the mutual positions of parallel $l$-edges.
Thus, $\T_{\mathrm s}\sim_{\mathrm I}\T_{\mathrm l}$, although their
knot intervals are different.
\end{example}

\begin{figure}[htbp]
\centering
\begin{tikzpicture}[line cap=round,line join=round]
  \begin{scope}[x=.8cm,y=.8cm]
    \fill[black!7] (1,1) rectangle (2,2);

    \foreach \x in {0,1,2,3}
      {\draw[line width=.6pt] (\x,0)--(\x,3);}
    \foreach \y in {0,1,2,3}
      {\draw[line width=.6pt] (0,\y)--(3,\y);}

    \draw[line width=1.1pt] (1.5,1)--(1.5,2);
    \draw[line width=1.1pt] (1,1.5)--(2,1.5);

    \foreach \x in {0,1,2,3}
      {\node[font=\scriptsize] at (\x,-.28) {$\x$};}
    \foreach \y in {0,1,2,3}
      {\node[font=\scriptsize,anchor=east]
        at (-.16,\y) {$\y$};}

    \node[font=\small] at (1.5,3.48) {$\T_{\mathrm s}$};
    \node[font=\scriptsize,fill=white,inner sep=1pt]
      at (1.5,1.78) {$R_{\mathrm s}$};
  \end{scope}

  \begin{scope}[shift={(5.2,0)},x=.8cm,y=.8cm]
    \fill[black!7] (2,1) rectangle (5,4);

    \foreach \x in {0,2,5,9}
      {\draw[line width=.6pt] (\x,0)--(\x,7);}
    \foreach \y in {0,1,4,7}
      {\draw[line width=.6pt] (0,\y)--(9,\y);}

    \draw[line width=1.1pt] (3.5,1)--(3.5,4);
    \draw[line width=1.1pt] (2,2.5)--(5,2.5);

    \foreach \x in {0,2,5,9}
      {\node[font=\scriptsize] at (\x,-.28) {$\x$};}
    \foreach \y in {0,1,4,7}
      {\node[font=\scriptsize,anchor=east]
        at (-.16,\y) {$\y$};}

    \node[font=\small] at (4.5,7.48) {$\T_{\mathrm l}$};
    \node[font=\scriptsize,fill=white,inner sep=1pt]
      at (3.5,3.2) {$R_{\mathrm l}$};
  \end{scope}
\end{tikzpicture}
\caption{Two structurally isomorphic hierarchical T-meshes.}
\label{fig:hb-isomorphic-meshes}
\end{figure}
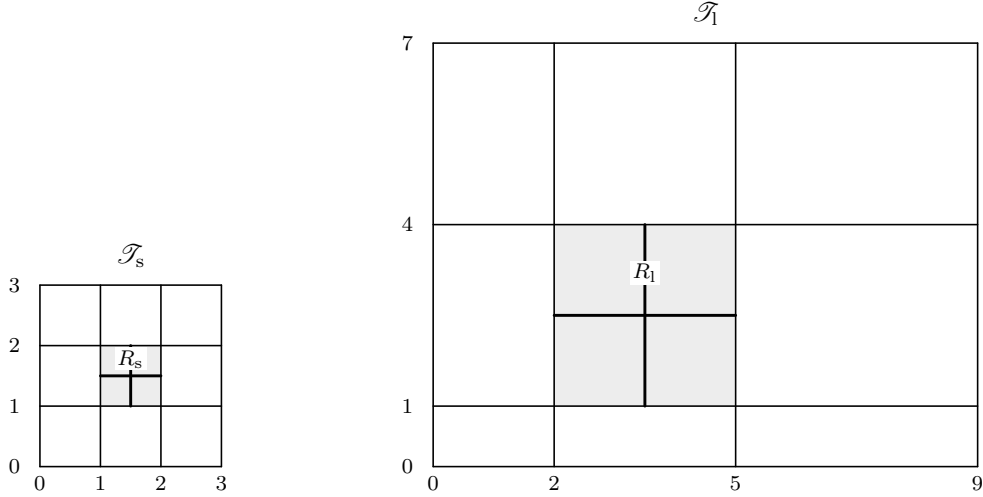

\begin{definition}[Dimensional stability;
\cite{HuangChen2025Classification}]\label{def:stability}
A T-mesh $\T\in U$ is \emph{dimensionally stable} for $S_d$ if
\[
  \dim S_d(\T_1)=\dim S_d(\T_2)
  \qquad\text{for all }\T_1,\T_2\in[\T]_{\mathrm I}.
\]
Otherwise, it is \emph{dimensionally unstable}.
\end{definition}

\begin{example}\label{ex:stable-class}
For $d=1$, the meshes in
Example~\ref{ex:structural-isomorphism} have dimension seventeen.
The values at the sixteen level-zero vertices and at the new
cross-vertex can be assigned independently. The values at the four
new T-nodes are determined by linear interpolation along the adjacent
unrefined edges. These values uniquely determine a continuous
bilinear spline on every cell. Hence,
\[
  \dim S_1(\widehat\T)=17
  \qquad\text{for every }\widehat\T\in[\T_{\mathrm s}]_{\mathrm I}.
\]
The examples in Sections~\ref{sec:vanishable-examples}
and~\ref{sec:examples} show that this invariance does not hold for
all hierarchical T-meshes when $d=3,4,5,6$.
\end{example}

\section{Construction-suitable T-meshes}\label{sec:construction-suitable}

In this section, we introduce construction-suitable T-meshes and
establish dimensional stability as a necessary condition for basis
construction. Following the viewpoint of
\cite{Huang2025BasisConstruction}, we distinguish the construction
of a linearly independent spline family from that of a basis on
an entire structurally isomorphic class.

Throughout this section, the degree $d$ and the admissible collection
$U$ are fixed. Each construction rule specifies all choices needed
to determine its functions. For hierarchical T-meshes, corresponding
mesh data include the prescribed levelwise subdivisions.

\begin{definition}[Mesh-induced spline construction]\label{def:alpha}
A rule $\alpha$ is called a \emph{mesh-induced spline construction}
on $U$ if it assigns to each $\T\in U$ a finite family
\begin{equation}\label{eq:alpha-mesh-map}
  \alpha(\T)
  =\{f^{\T}_{\lambda}:\lambda\in\Lambda_{\T}\}
  \subset S_d(\T),
\end{equation}
whose functions are determined by the mesh data and are linearly
independent.
\end{definition}

The family $\alpha(\T)$ need not span $S_d(\T)$. Moreover, a
construction defined on individual T-meshes need not give
corresponding families on structurally isomorphic T-meshes.

Two constructed families are said to \emph{correspond} under a
structurally isomorphic map if their index sets are in one-to-one
correspondence and the corresponding functions are defined by the
same formulas using corresponding mesh data. In particular, their
supporting cells, refinement levels and local knot indices
correspond. The knot values are taken from the respective meshes;
the resulting functions need not have the same coefficients or
physical supports.

\begin{definition}[Well-defined basis construction]\label{def:beta}
Let $\alpha$ be a mesh-induced spline construction on $U$. Its
restriction to $[\T]_{\mathrm I}$ is called a \emph{well-defined
basis construction}, denoted by $\beta$, if the following
conditions hold:
\begin{enumerate}[label=\textup{(\roman*)}]
  \item
  For every $\widehat\T\in[\T]_{\mathrm I}$,
  \[
    \operatorname{span}\alpha(\widehat\T)
    =S_d(\widehat\T).
  \]

  \item
  For any $\T_1,\T_2\in[\T]_{\mathrm I}$, the families
  $\alpha(\T_1)$ and $\alpha(\T_2)$ correspond under every
  structurally isomorphic map compatible with the prescribed
  mesh data,
  \[
    \varphi:\mathcal E(\T_1)\longrightarrow\mathcal E(\T_2).
  \]
\end{enumerate}
In this case,
$\beta(\widehat\T)=\alpha(\widehat\T)$ for every
$\widehat\T\in[\T]_{\mathrm I}$.
\end{definition}

Since the functions constructed by $\alpha$ are linearly independent,
condition~(i) ensures that they form a basis. Condition~(ii) requires
the construction to be independent of the chosen mesh within the
class. We use $\beta([\T]_{\mathrm I})$ to denote this common
construction, rather than a single family of functions on a fixed
domain.

\begin{definition}[Construction-suitable T-mesh]
\label{def:construction-suitable}
A T-mesh $\T\in U$ is called \emph{construction-suitable for $S_d$}
if a well-defined basis construction $\beta$ exists on
$[\T]_{\mathrm I}$. When the construction is specified, we say that
$\T$ is construction-suitable \emph{with respect to $\beta$}.
\end{definition}

The following proposition relates construction-suitability to the
stability of the dimension.

\begin{proposition}\label{prop:stability-basis-map}
If $\T$ is construction-suitable for $S_d$, then the dimension of
$S_d$ is stable on $[\T]_{\mathrm I}$.
\end{proposition}

\begin{proof}
Let $\beta$ be a well-defined basis construction on
$[\T]_{\mathrm I}$, and let $\T_1,\T_2\in[\T]_{\mathrm I}$.
By Definition~\ref{def:beta}, the two constructed families have
the same number of functions. Since both families are bases,
\[
  \dim S_d(\T_1)
  =|\beta(\T_1)|
  =|\beta(\T_2)|
  =\dim S_d(\T_2).
\]
Thus, the dimension is constant on $[\T]_{\mathrm I}$.
\end{proof}

Proposition~\ref{prop:stability-basis-map} gives a necessary condition
for construction-suitability. For a proposed construction, linear
independence and completeness must still be established. When the
dimension is unstable, a basis can be constructed separately on
each T-mesh, but these bases cannot arise from a common construction
satisfying Definition~\ref{def:beta}.

We next give an example to illustrate the correspondence between
constructed functions on two structurally isomorphic T-meshes.

\begin{example}[A bilinear HB construction]\label{ex:hb-alpha}
Consider the meshes $\T_{\mathrm s}$ and $\T_{\mathrm l}$ in
Example~\ref{ex:structural-isomorphism}, and let $d=1$. For
$\nu\in\{\mathrm s,\mathrm l\}$, denote by
$B^{0,\nu}_{ij}$, $0\le i,j\le3$, the open-clamped bilinear
tensor-product B-splines on the initial mesh. The level-one
B-splines are defined on the tensor-product mesh obtained by
cross-subdividing every cell of the initial mesh.

The one-level HB construction retains the coarse B-splines whose
supports are not contained in the refined region $R_\nu$ and
selects the fine B-splines whose supports are contained in $R_\nu$
\cite{ForseyBartels1988}. In this example, no coarse support is
contained in $R_\nu$, whereas exactly one level-one B-spline has
support contained in $R_\nu$. Denote this function by
$B^{1,\nu}_{\star}$. The resulting family is
\begin{equation}\label{eq:hb-example-family}
  \alpha_{\mathrm{HB}}(\T_\nu)
  =\{B^{0,\nu}_{ij}:0\le i,j\le3\}
   \cup\{B^{1,\nu}_{\star}\}.
\end{equation}

All seventeen functions belong to $S_1(\T_\nu)$. Indeed, the coarse
B-splines remain continuous and bilinear on each cell, while
$B^{1,\nu}_{\star}$ is continuous, piecewise bilinear on $R_\nu$,
and zero on $\partial R_\nu$ and outside $R_\nu$. To prove linear
independence, consider a linear relation among these functions.
Evaluation at the sixteen initial grid vertices gives zero
coefficients for all coarse B-splines, since they are nodal at
these vertices and the fine B-spline vanishes there. Evaluation
at the center of $R_\nu$ then gives a zero coefficient for
$B^{1,\nu}_{\star}$.

Let $I_i^\nu$ and $J_j^\nu$ denote the univariate support intervals
of the coarse B-splines. Then
\begin{equation}\label{eq:hb-example-supports}
  \operatorname{supp}B^{0,\nu}_{ij}
  =I_i^\nu\times J_j^\nu,
  \qquad
  \operatorname{supp}B^{1,\nu}_{\star}=R_\nu.
\end{equation}
For the two meshes, these intervals are
\begin{align*}
  (I^{\mathrm s}_0,I^{\mathrm s}_1,
   I^{\mathrm s}_2,I^{\mathrm s}_3)
  &=([0,1],[0,2],[1,3],[2,3]),\\
  (I^{\mathrm l}_0,I^{\mathrm l}_1,I^{\mathrm l}_2,I^{\mathrm l}_3)
  &=([0,2],[0,5],[2,9],[5,9]),\\
  (J^{\mathrm l}_0,J^{\mathrm l}_1,J^{\mathrm l}_2,J^{\mathrm l}_3)
  &=([0,1],[0,4],[1,7],[4,7]),
\end{align*}
with $J_j^{\mathrm s}=I_j^{\mathrm s}$ for $0\le j\le3$.
The supports are shown in Figure~\ref{fig:hb-support-atlas}.
In each panel, the upper array contains the sixteen level-zero
supports, and the lower diagram shows the selected level-one
support.

The construction and the independence argument apply to every
mesh in $[\T_{\mathrm s}]_{\mathrm I}$. By
Example~\ref{ex:stable-class}, the full spline space has dimension
seventeen throughout this class. Hence the constructed family is
a basis on each mesh. Moreover, the indices, levels and supporting
cells correspond under structurally isomorphic maps, and the same
B-spline formulas use the corresponding knots. Therefore,
$\alpha_{\mathrm{HB}}$ defines a well-defined basis construction
$\beta_{\mathrm{HB}}$, and the meshes in this class are
construction-suitable for $S_1$.
\end{example}

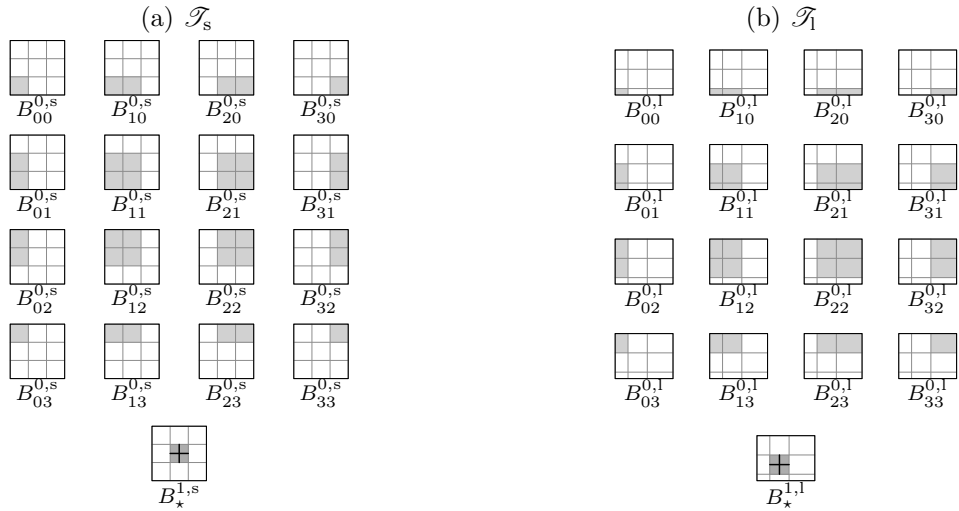
\begin{figure}[htbp]
\centering
\begin{tikzpicture}[line cap=round,line join=round]
  \node[font=\small] at (2.2,4.75) {(a) $\T_{\mathrm s}$};

  \foreach \i/\xa/\xb in {0/0/1,1/0/2,2/1/3,3/2/3}{
    \foreach \j/\ya/\yb in {0/0/1,1/0/2,2/1/3,3/2/3}{
      \begin{scope}[
        shift={({1.25*\i},{1.25*(3-\j)})},
        x=.24cm,y=.24cm
      ]
        \fill[black!18] (\xa,\ya) rectangle (\xb,\yb);

        \foreach \x in {0,1,2,3}{
          \draw[black!45,line width=.25pt]
            (\x,0)--(\x,3);
        }
        \foreach \y in {0,1,2,3}{
          \draw[black!45,line width=.25pt]
            (0,\y)--(3,\y);
        }

        \draw[line width=.4pt] (0,0) rectangle (3,3);
        \node[font=\scriptsize] at (1.5,-.8)
          {$B^{0,\mathrm s}_{\i\j}$};
      \end{scope}
    }
  }

  \begin{scope}[shift={(1.875,-1.35)},x=.24cm,y=.24cm]
    \fill[black!32] (1,1) rectangle (2,2);

    \foreach \x in {0,1,2,3}{
      \draw[black!45,line width=.25pt] (\x,0)--(\x,3);
    }
    \foreach \y in {0,1,2,3}{
      \draw[black!45,line width=.25pt] (0,\y)--(3,\y);
    }

    \draw[line width=.4pt] (0,0) rectangle (3,3);
    \draw[line width=.55pt] (1.5,1)--(1.5,2);
    \draw[line width=.55pt] (1,1.5)--(2,1.5);

    \node[font=\scriptsize] at (1.5,-.8)
      {$B^{1,\mathrm s}_{\star}$};
  \end{scope}

  \begin{scope}[shift={(8,0)}]
    \node[font=\small] at (2.2,4.75) {(b) $\T_{\mathrm l}$};

    \foreach \i/\xa/\xb in {0/0/2,1/0/5,2/2/9,3/5/9}{
      \foreach \j/\ya/\yb in {0/0/1,1/0/4,2/1/7,3/4/7}{
        \begin{scope}[
          shift={({1.25*\i},{1.25*(3-\j)})},
          x=.085cm,y=.085cm
        ]
          \fill[black!18] (\xa,\ya) rectangle (\xb,\yb);

          \foreach \x in {0,2,5,9}{
            \draw[black!45,line width=.25pt]
              (\x,0)--(\x,7);
          }
          \foreach \y in {0,1,4,7}{
            \draw[black!45,line width=.25pt]
              (0,\y)--(9,\y);
          }

          \draw[line width=.4pt] (0,0) rectangle (9,7);
          \node[font=\scriptsize] at (4.5,-2.25)
            {$B^{0,\mathrm l}_{\i\j}$};
        \end{scope}
      }
    }

    \begin{scope}[shift={(1.875,-1.35)},x=.085cm,y=.085cm]
      \fill[black!32] (2,1) rectangle (5,4);

      \foreach \x in {0,2,5,9}{
        \draw[black!45,line width=.25pt] (\x,0)--(\x,7);
      }
      \foreach \y in {0,1,4,7}{
        \draw[black!45,line width=.25pt] (0,\y)--(9,\y);
      }

      \draw[line width=.4pt] (0,0) rectangle (9,7);
      \draw[line width=.55pt] (3.5,1)--(3.5,4);
      \draw[line width=.55pt] (2,2.5)--(5,2.5);

      \node[font=\scriptsize] at (4.5,-2.25)
        {$B^{1,\mathrm l}_{\star}$};
    \end{scope}
  \end{scope}
\end{tikzpicture}
\caption{Supports of the bilinear HB-splines.}
\label{fig:hb-support-atlas}
\end{figure}

In Section~\ref{sec:instability-assumption}, we examine whether
hierarchical refinement ensures dimensional stability, first allowing
vanishable T $l$-edges and then excluding them at every level.

\section{Instability examples and a mesh assumption}
\label{sec:instability-assumption}

By Proposition~\ref{prop:stability-basis-map}, dimensional stability is
necessary for the proposed basis construction. In this section, we
examine two refinement conditions for hierarchical T-meshes.
Section~\ref{sec:vanishable-examples} considers unrestricted
hierarchical refinement, whereas Section~\ref{sec:examples} excludes
vanishable T $l$-edges at every level. Four examples show that neither
condition ensures the stability of the dimension. We then describe
refined regions by template translations and introduce a mesh
assumption for the subsequent $\PHtwoT$-spline
construction. The conformality matrix reductions are described in
Appendix~\ref{app:conformality}. The four dimension calculations
are given in Appendices~\ref{app:d3}--\ref{app:d6}, and the proofs
of the template conditions are collected in
Appendix~\ref{app:templates}.

To specify the subdivisions, we index the cells from zero in each
coordinate direction. At level $k$, the indices refer to the
tensor-product grid obtained by $k$ uniform cross subdivisions of
$\T^0$. Subdividing a level-$k$ cell with index $(i,j)$ gives four
cells with indices $(2i+\varepsilon,2j+\eta)$, where
$\varepsilon,\eta\in\{0,1\}$. We use $\cR^{(d)}_k$ to denote the
indices of the level-$(k-1)$ cells selected for subdivision in the
example of degree $d$. Their geometric coordinates are determined
by the initial grid and the midpoint subdivision rule.

\subsection{Unrestricted hierarchical refinement}
\label{sec:vanishable-examples}

Here, \emph{unrestricted} means that no additional condition is
imposed on the cells selected for hierarchical cross subdivision.
In particular, vanishable T $l$-edges may be generated at any level.
By Definition~\ref{def:vanishable}, the conformality equations on
such an edge force all its vertex cofactors to zero. The edge thus
provides no free cofactor, and a refinement may insert new edges without
increasing the dimension of the full spline space. The following examples illustrate
this situation and show that the dimension may depend on the
geometric information of the hierarchical T-mesh.

\subsubsection{A bicubic example}\label{sec:d3}

Let $\tau>0$, and consider the initial tensor-product mesh with
grid-line coordinates
\begin{equation}\label{eq:d3knots}
  X(\tau)=(0,1,1+\tau,2+\tau,3+\tau,4+\tau),
  \qquad Y=(0,1,2,3,4,5).
\end{equation}
Starting from this $5\times5$ cell mesh, perform three levels of
cross subdivision with
\begin{align}
  \cR^{(3)}_1
  & =\{(1,3),(2,2),(2,3),(3,1),(3,2)\},
    \label{eq:d3R1}\\
  \cR^{(3)}_2
  & =\{(5,4),(5,5),(5,6),(6,4)\},
    \label{eq:d3R2}\\
  \cR^{(3)}_3
  & =\{(10,8),(10,9),(10,10),(11,9),\nonumber\\
  & \hspace{12mm}(11,10),(11,11),(11,12),(11,13)\}.
    \label{eq:d3R3}
\end{align}
The resulting meshes are denoted by $\T_3^k(\tau)$, $0\le k\le3$,
with $\T_3(\tau)=\T_3^3(\tau)$. All cells selected at the second
and third levels have been generated at the preceding level.
The numbers of cells are
\[
  25\longrightarrow40\longrightarrow52\longrightarrow76.
\]

Figure~\ref{fig:d3hierarchy} shows the three refinement levels for
$\tau=1$ and the final mesh for $\tau=2$. The shaded cells indicate
the region subdivided at the displayed level. Varying $\tau$ changes
only the second initial knot interval in the $x$-direction; the
subdivision indices and the midpoint rule remain unchanged.
Consequently, the meshes $\T_3(\tau)$ are structurally isomorphic
for $\tau>0$. At each refinement level, at least one new T $l$-edge
has only three vertices and is therefore vanishable for $S_3$.

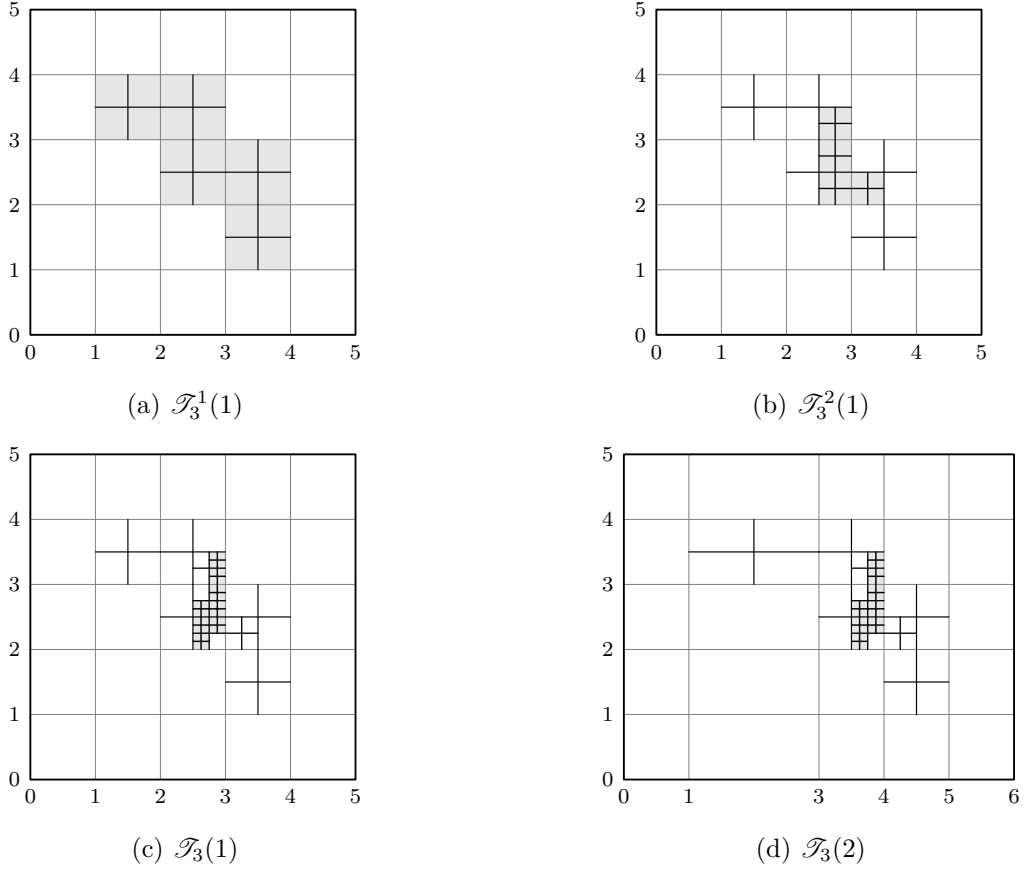
\begin{figure}[htbp]
\centering
\def\BicubicRone{1/3,2/2,2/3,3/1,3/2}
\def\BicubicRtwo{5/4,5/5,5/6,6/4}
\def\BicubicRthree{10/8,10/9,10/10,11/9,11/10,11/11,11/12,11/13}
\newcommand{\DrawBicubicStep}[2]{  \begin{tikzpicture}[
    x=.86cm,y=.86cm,line cap=round,line join=round,
    declare function={bix(\z)=\z+(#2-1)*min(1,max(0,\z-1));}
  ]
    \ifnum#1=1
      \foreach \i/\j in \BicubicRone{
        \fill[black!10] ({bix(\i)},\j)
          rectangle ({bix(\i+1)},\j+1);
      }
    \fi
    \ifnum#1=2
      \foreach \i/\j in \BicubicRtwo{
        \fill[black!10] ({bix(\i/2)},\j/2)
          rectangle ({bix((\i+1)/2)},{(\j+1)/2});
      }
    \fi
    \ifnum#1=3
      \foreach \i/\j in \BicubicRthree{
        \fill[black!10] ({bix(\i/4)},\j/4)
          rectangle ({bix((\i+1)/4)},{(\j+1)/4});
      }
    \fi
    \foreach \i in {0,...,5}{
      \draw[black!50,line width=.35pt]
        ({bix(\i)},0)--({bix(\i)},5);
      \draw[black!50,line width=.35pt]
        (0,\i)--({bix(5)},\i);
      \pgfmathsetmacro{\tickx}{bix(\i)}
      \node[font=\scriptsize,below] at (\tickx,0)
        {$\pgfmathprintnumber{\tickx}$};
      \node[font=\scriptsize,left] at (0,\i) {$\i$};
    }
    \foreach \i/\j in \BicubicRone{
      \draw[line width=.45pt]
        ({bix(\i+.5)},\j)--({bix(\i+.5)},\j+1);
      \draw[line width=.45pt]
        ({bix(\i)},\j+.5)--({bix(\i+1)},\j+.5);
    }
    \ifnum#1>1
      \foreach \i/\j in \BicubicRtwo{
        \draw[line width=.5pt]
          ({bix((\i+.5)/2)},\j/2)
          --({bix((\i+.5)/2)},{(\j+1)/2});
        \draw[line width=.5pt]
          ({bix(\i/2)},{(\j+.5)/2})
          --({bix((\i+1)/2)},{(\j+.5)/2});
      }
    \fi
    \ifnum#1>2
      \foreach \i/\j in \BicubicRthree{
        \draw[line width=.55pt]
          ({bix((\i+.5)/4)},\j/4)
          --({bix((\i+.5)/4)},{(\j+1)/4});
        \draw[line width=.55pt]
          ({bix(\i/4)},{(\j+.5)/4})
          --({bix((\i+1)/4)},{(\j+.5)/4});
      }
    \fi
    \draw[line width=.7pt] (0,0) rectangle ({bix(5)},5);
  \end{tikzpicture}}
\begin{minipage}[t]{.48\linewidth}
  \centering
  \DrawBicubicStep{1}{1}
  \par\smallskip{\small (a) $\T_3^1(1)$}
\end{minipage}\hfill
\begin{minipage}[t]{.48\linewidth}
  \centering
  \DrawBicubicStep{2}{1}
  \par\smallskip{\small (b) $\T_3^2(1)$}
\end{minipage}
\par\medskip
\begin{minipage}[t]{.48\linewidth}
  \centering
  \DrawBicubicStep{3}{1}
  \par\smallskip{\small (c) $\T_3(1)$}
\end{minipage}\hfill
\begin{minipage}[t]{.48\linewidth}
  \centering
  \DrawBicubicStep{3}{2}
  \par\smallskip{\small (d) $\T_3(2)$}
\end{minipage}
\caption{A three-level bicubic example.}
\label{fig:d3hierarchy}
\end{figure}

\begin{theorem}\label{thm:d3}
For the hierarchical T-mesh $\T_3(\tau)$ defined above,
\begin{equation}\label{eq:d3dimension}
  \dim S_3(\T_3(\tau))=
  \begin{cases}
    66, & \tau=1,\\
    65, & \tau>0,\ \tau\ne1.
  \end{cases}
\end{equation}
Thus, the dimension of $S_3$ is not stable on
$[\T_3(1)]_{\mathrm I}$.
\end{theorem}

The proof of Theorem~\ref{thm:d3}, including the reduced matrix
and its exact rank calculation, is given in Appendix~\ref{app:d3}.
The dimensions at successive levels for $\tau=1,2$ are listed in
Table~\ref{tab:d3stages}. The first two refinements add edges without
increasing the dimension. At the third level, the dimension
increases by two when $\tau=1$ and by one when $\tau=2$.

\begin{table}[htbp]
\centering
\caption{Dimensions at the levels of the bicubic example.}
\label{tab:d3stages}
\begin{tabular}{@{}cccc@{}}
\toprule
Level $k$ & Number of cells
  & $\dim S_3(\T_3^k(1))$ & $\dim S_3(\T_3^k(2))$\\
\midrule
0 & 25 & 64 & 64\\
1 & 40 & 64 & 64\\
2 & 52 & 64 & 64\\
3 & 76 & 66 & 65\\
\bottomrule
\end{tabular}
\end{table}

\subsubsection{A biquartic example}\label{sec:d4}

We next consider the hierarchical T-mesh in
\cite[Example~4.1 and Figure~10]{HuangChen2026}.
For $0<\tau<2$, take the initial grid-line coordinates
\begin{equation}\label{eq:d4knots}
  X(\tau)=(0,\tau,2,3,4,5,6),
  \qquad Y=(0,1,2,3,4,5,6).
\end{equation}
Cross-subdivide the eight cells with indices
\begin{equation}\label{eq:d4mask}
  \begin{split}
    \cR^{(4)}_1=\{&(2,4),(1,3),(2,3),(3,3),\\
                 &(2,2),(3,2),(4,2),(3,1)\}.
  \end{split}
\end{equation}
The resulting mesh, denoted by $\T_4(\tau)$, has sixty cells.
The same subdivision is used for all $\tau$, so these meshes are
structurally isomorphic.

Figure~\ref{fig:d4mesh} shows $\T_4(1)$ and $\T_4(3/2)$.
The four dashed T $l$-edges each contain three vertices and are
vanishable for $S_4$. In panel~(a), these are
$\overline{v_2v_3}$, $\overline{v_4v_8}$,
$\overline{v_9v_{13}}$ and $\overline{v_{14}v_{15}}$.

\begin{figure}[htbp]
\centering
\newcommand{\DrawBiquarticMesh}[2]{  \begin{tikzpicture}[x=.9cm,y=.9cm,line cap=round,line join=round]
    \pgfmathsetmacro{\midleft}{(#1+2)/2}
    \foreach \x in {0,#1,2,3,4,5,6}{
      \draw[black!50,line width=.4pt] (\x,0)--(\x,6);
      \node[font=\scriptsize,below] at (\x,0)
        {$\pgfmathprintnumber{\x}$};
    }
    \foreach \y in {0,...,6}{
      \draw[black!50,line width=.4pt] (0,\y)--(6,\y);
      \node[font=\scriptsize,left] at (0,\y) {$\y$};
    }
    \draw[line width=.65pt] (2,2.5)--(5,2.5);
    \draw[line width=.65pt] (#1,3.5)--(4,3.5);
    \draw[line width=.65pt] (2.5,2)--(2.5,5);
    \draw[line width=.65pt] (3.5,1)--(3.5,4);
    \draw[line width=.65pt,densely dashed] (3,1.5)--(4,1.5);
    \draw[line width=.65pt,densely dashed] (2,4.5)--(3,4.5);
    \draw[line width=.65pt,densely dashed]
      (\midleft,3)--(\midleft,4);
    \draw[line width=.65pt,densely dashed] (4.5,2)--(4.5,3);
    \draw[line width=.7pt] (0,0) rectangle (6,6);
    \ifnum#2=1
      \foreach \i/\x/\y in {
        1/2.5/5,2/2/4.5,3/3/4.5,4/1.5/4,
        5/3.5/4,6/1/3.5,7/4/3.5,8/1.5/3,
        9/4.5/3,10/2/2.5,11/5/2.5,12/2.5/2,
        13/4.5/2,14/3/1.5,15/4/1.5,16/3.5/1
      }{
        \fill (\x,\y) circle (1.4pt);
        \node[font=\scriptsize,above right,inner sep=1.5pt]
          at (\x,\y) {$v_{\i}$};
      }
    \fi
  \end{tikzpicture}}
\begin{minipage}[t]{.48\linewidth}
  \centering
  \DrawBiquarticMesh{1}{1}
  \par\smallskip{\small (a) $\T_4(1)$}
\end{minipage}\hfill
\begin{minipage}[t]{.48\linewidth}
  \centering
  \DrawBiquarticMesh{1.5}{0}
  \par\smallskip{\small (b) $\T_4(3/2)$}
\end{minipage}
\caption{A one-level biquartic example.}
\label{fig:d4mesh}
\end{figure}
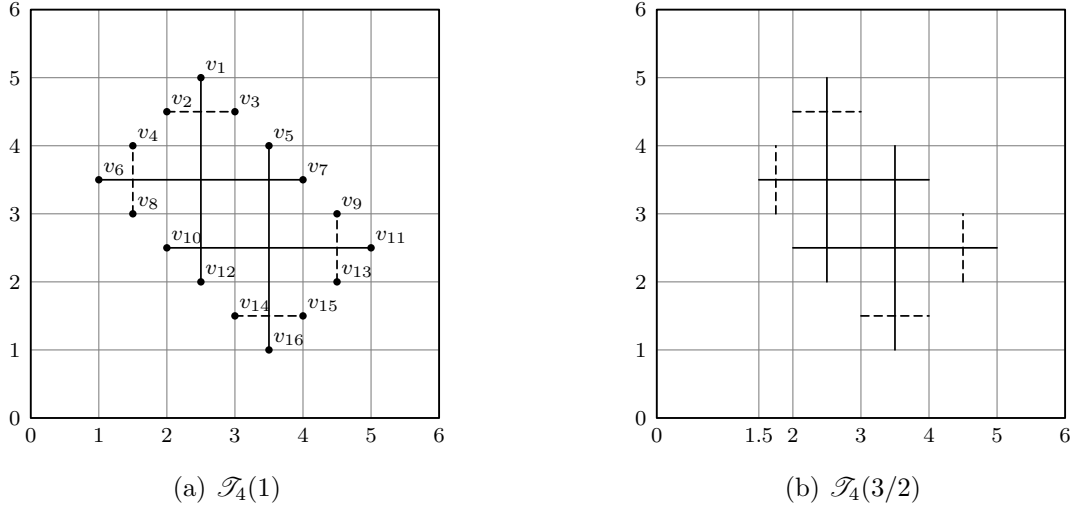

\begin{theorem}\label{thm:d4}
For the hierarchical T-mesh $\T_4(\tau)$ defined above,
\begin{equation}\label{eq:d4dimension}
  \dim S_4(\T_4(\tau))=
  \begin{cases}
    101, & \tau=1,\\
    100, & 0<\tau<2,\ \tau\ne1.
  \end{cases}
\end{equation}
Thus, the dimension of $S_4$ is not stable on
$[\T_4(1)]_{\mathrm I}$.
\end{theorem}

The proof of Theorem~\ref{thm:d4} is given in Appendix~\ref{app:d4},
where the dimension calculation is reduced to an explicit
$4\times4$ matrix. The initial tensor-product space has dimension
$100$ for every $0<\tau<2$. Hence this refinement adds one degree of freedom
when $\tau=1$ and none otherwise.

Theorems~\ref{thm:d3} and~\ref{thm:d4} show that unrestricted
hierarchical refinement does not ensure dimensional stability.
The examples also exhibit refinements that introduce vanishable
T $l$-edges without increasing the dimension. This motivates imposing a
restriction that excludes such edges at each level.

\subsection{Refinement without vanishable T \texorpdfstring{$l$}{l}-edges}
\label{sec:examples}
\label{sec:no-vanishable-examples}

We now require every newly generated T $l$-edge to be non-vanishable
immediately after subdivision. Equivalently, each such edge must
contain at least $d+2$ vertices. The next two examples satisfy this
condition and, in addition, strictly increase the dimension at
every refinement level. Nevertheless, their spline-space dimensions
are not stable.

\subsubsection{An example of degree five}\label{sec:d5}

Start with a $7\times7$ tensor-product cell mesh. At the first and
second refinement levels, select
\begin{align}
  \cR^{(5)}_1
    &=\{1,2,3\}^2\cup\{3,4,5\}^2,
      \label{eq:d5R1}\\
  \cR^{(5)}_2
    &=\{4,5,6\}^2\cup\{6,7,8\}^2,
      \label{eq:d5R2}
\end{align}
where $A^2=A\times A$ for a set of integers $A$.
Each selected region consists of two $3\times3$ cell submeshes
that overlap in one cell. The second region contains only
level-one cells generated by the first subdivision. There are
seventeen selected cells at each level, and the cell numbers are
\[
  49\longrightarrow100\longrightarrow151.
\]

Let $x_0<\cdots<x_7$ and $y_0<\cdots<y_7$ be the initial grid-line
coordinates. Consider the two meshes
\begin{equation}\label{eq:d5knots}
  \begin{aligned}
    \T_{5,A}:&\quad x_i=y_i=i,\qquad 0\le i\le7,\\
    \T_{5,B}:&\quad x_1=\frac{11}{10},\quad
      x_i=i\ (i\ne1),\quad y_i=i.
  \end{aligned}
\end{equation}
All subsequent coordinates are determined by midpoint subdivision.
The two meshes have the same levelwise subdivisions and are
structurally isomorphic. We write $\T_{5,A}^k$ and $\T_{5,B}^k$
for the meshes at level $k$, so that $\T_{5,A}=\T_{5,A}^2$ and
$\T_{5,B}=\T_{5,B}^2$.

The mesh $\T_{5,A}$ is shown in Figure~\ref{fig:d5mesh}. Blue and
red line segments are introduced at levels one and two,
respectively. Dashed rectangles bound the two submeshes selected
at the first level, and dotted rectangles bound those selected
at the second level. Every newly generated T $l$-edge contains at
least seven vertices, so no vanishable T $l$-edges are introduced
for $d=5$.

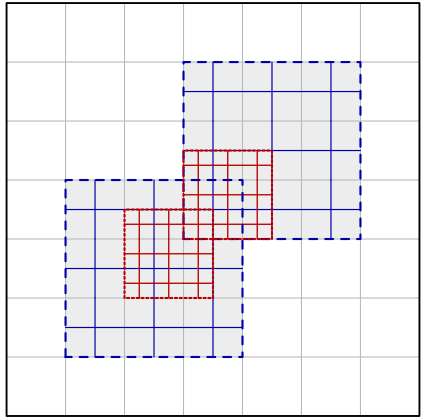
\begin{figure}[htbp]
\centering
\begin{tikzpicture}[scale=.78,line cap=round,line join=round]
  \fill[black!7] (1,1) rectangle (4,4);
  \fill[black!7] (3,3) rectangle (6,6);
  \foreach \x in {0,...,7}{\draw[gray!55] (\x,0)--(\x,7);}
  \foreach \y in {0,...,7}{\draw[gray!55] (0,\y)--(7,\y);}
  \foreach \i/\j in {
    1/1,1/2,1/3,2/1,2/2,2/3,3/1,3/2,3/3,
    3/4,3/5,4/3,4/4,4/5,5/3,5/4,5/5
  }{
    \draw[blue!65!black,line width=.45pt]
      (\i+.5,\j)--(\i+.5,\j+1);
    \draw[blue!65!black,line width=.45pt]
      (\i,\j+.5)--(\i+1,\j+.5);
  }
  \foreach \i/\j in {
    4/4,4/5,4/6,5/4,5/5,5/6,6/4,6/5,6/6,
    6/7,6/8,7/6,7/7,7/8,8/6,8/7,8/8
  }{
    \draw[red!72!black,line width=.5pt]
      ({(\i+.5)/2},\j/2)--({(\i+.5)/2},{(\j+1)/2});
    \draw[red!72!black,line width=.5pt]
      (\i/2,{(\j+.5)/2})--({(\i+1)/2},{(\j+.5)/2});
  }
  \draw[blue!65!black,thick,dashed] (1,1) rectangle (4,4);
  \draw[blue!65!black,thick,dashed] (3,3) rectangle (6,6);
  \draw[red!72!black,thick,densely dotted] (2,2) rectangle (3.5,3.5);
  \draw[red!72!black,thick,densely dotted] (3,3) rectangle (4.5,4.5);
  \draw[line width=.7pt] (0,0) rectangle (7,7);
\end{tikzpicture}
\caption{The hierarchical T-mesh $\T_{5,A}$.}
\label{fig:d5mesh}
\end{figure}

The CNDC decomposition and exact rank certificates are given in
Appendix~\ref{app:d5}. They yield the dimensions
\begin{equation}\label{eq:d5sequence}
  \begin{aligned}
    \T_{5,A}:&\quad144\longrightarrow147\longrightarrow150,\\
    \T_{5,B}:&\quad144\longrightarrow147\longrightarrow149.
  \end{aligned}
\end{equation}
Thus, both refinement levels add degrees of freedom, but
\[
  \dim S_5(\T_{5,A})=150\ne149=\dim S_5(\T_{5,B}).
\]
The dimension is therefore unstable, although every refinement
level is no-vanishable.

\subsubsection{An example of degree six}\label{sec:d6}

Start with a $10\times10$ tensor-product cell mesh and select
\begin{align}
  \cR^{(6)}_1
    &=\{1,2,3,4\}^2\cup\{4,5,6,7\}^2,
      \label{eq:d6R1}\\
  \cR^{(6)}_2
    &=\{6,7,8,9\}^2\cup\{9,10,11,12\}^2
      \label{eq:d6R2}
\end{align}
at the first and second refinement levels, respectively.
Each region consists of two $4\times4$ cell submeshes with a
one-cell overlap. All cells selected at the second level are
level-one cells. The cell numbers are
\[
  100\longrightarrow193\longrightarrow286.
\]

Consider the initial grid-line coordinates
\begin{equation}\label{eq:d6knots}
  \begin{aligned}
    \T_{6,A}:&\quad x_i=y_i=i,\qquad 0\le i\le10,\\
    \T_{6,B}:&\quad x_1=\frac{11}{10},\quad
      x_i=i\ (i\ne1),\quad y_i=i.
  \end{aligned}
\end{equation}
The same midpoint subdivisions give two structurally isomorphic
hierarchical T-meshes. As before, $\T_{6,A}^k$ and $\T_{6,B}^k$
denote the level-$k$ meshes, with the final meshes at level two.
Figure~\ref{fig:d6mesh} shows $\T_{6,A}$,
using the same line conventions as Figure~\ref{fig:d5mesh}.
Every new T $l$-edge has at least nine vertices and is therefore
non-vanishable for $d=6$.

\begin{figure}[htbp]
\centering
\begin{tikzpicture}[scale=.56,line cap=round,line join=round]
  \fill[black!7] (1,1) rectangle (5,5);
  \fill[black!7] (4,4) rectangle (8,8);
  \foreach \x in {0,...,10}{\draw[gray!55] (\x,0)--(\x,10);}
  \foreach \y in {0,...,10}{\draw[gray!55] (0,\y)--(10,\y);}
  \foreach \a in {1,4}{
    \foreach \i in {0,...,3}{\foreach \j in {0,...,3}{
      \draw[blue!65!black,line width=.45pt]
        (\a+\i+.5,\a+\j)--(\a+\i+.5,\a+\j+1);
      \draw[blue!65!black,line width=.45pt]
        (\a+\i,\a+\j+.5)--(\a+\i+1,\a+\j+.5);
    }}
  }
  \foreach \a in {6,9}{
    \foreach \i in {0,...,3}{\foreach \j in {0,...,3}{
      \draw[red!72!black,line width=.5pt]
        ({(\a+\i+.5)/2},{(\a+\j)/2})
        --({(\a+\i+.5)/2},{(\a+\j+1)/2});
      \draw[red!72!black,line width=.5pt]
        ({(\a+\i)/2},{(\a+\j+.5)/2})
        --({(\a+\i+1)/2},{(\a+\j+.5)/2});
    }}
  }
  \draw[blue!65!black,thick,dashed] (1,1) rectangle (5,5);
  \draw[blue!65!black,thick,dashed] (4,4) rectangle (8,8);
  \draw[red!72!black,thick,densely dotted] (3,3) rectangle (5,5);
  \draw[red!72!black,thick,densely dotted] (4.5,4.5) rectangle (6.5,6.5);
  \draw[line width=.7pt] (0,0) rectangle (10,10);
\end{tikzpicture}
\caption{The hierarchical T-mesh $\T_{6,A}$.}
\label{fig:d6mesh}
\end{figure}
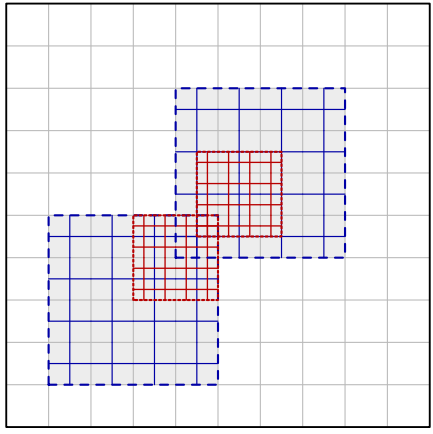

The detailed calculation in Appendix~\ref{app:d6} gives the
dimensions at successive levels as
\begin{equation}\label{eq:d6sequence}
  \begin{aligned}
    \T_{6,A}:&\quad256\longrightarrow267\longrightarrow277,\\
    \T_{6,B}:&\quad256\longrightarrow267\longrightarrow276.
  \end{aligned}
\end{equation}
Again, every refinement level strictly increases the dimension,
but the final dimensions are different.

The results of the four examples are summarized in
Table~\ref{tab:verified-dimensions}. Here, ``vanishable'' refers to
the presence of vanishable T $l$-edges at some refinement level.
The two dimensions in each row correspond to the two meshes
specified in the relevant example.

\begin{table}[htbp]
\centering
\caption{Dimensions in the four instability examples.}
\label{tab:verified-dimensions}
\begin{tabular}{@{}ccccc@{}}
\toprule
$d$ & Refinement levels & Number of cells
  & Vanishable T $l$-edges & Dimensions\\
\midrule
3 & 3 & 76  & Yes & $66,65$\\
4 & 1 & 60  & Yes & $101,100$\\
5 & 2 & 151 & No  & $150,149$\\
6 & 2 & 286 & No  & $277,276$\\
\bottomrule
\end{tabular}
\end{table}

\begin{theorem}\label{thm:main}
Excluding vanishable T $l$-edges at every refinement level is not
a sufficient condition for dimensional stability over hierarchical
T-meshes in general. In particular, the examples for $d=5,6$
are dimensionally unstable even though every refinement level
is no-vanishable and strictly increases the dimension.
\end{theorem}

Thus, excluding vanishable T $l$-edges does not remove the
obstruction to basis construction identified in
Proposition~\ref{prop:stability-basis-map}. A stronger refinement
condition is needed to ensure dimensional stability on the
mesh class used for the subsequent construction.

\subsection{Template translations and the mesh assumption}
\label{sec:template-refinement}
\label{sec:channel}
\label{sec:stable-assumption}

We now describe a refined region as a union of translated copies
of a fixed cell template. The translations are taken in the
cell-index grid introduced at the beginning of this section.
For a nonuniform mesh, the corresponding geometric submeshes
need not be Euclidean translates. This description therefore
restricts the arrangement of the selected cells without imposing
equal knot intervals.

\begin{definition}[Translated template]\label{def:template}
For an integer $m\ge1$, an \emph{$m\times m$ template} consists
of $m\times m$ consecutive cells in the two coordinate directions.
Its translate with initial index $(a,b)$ is denoted by
\begin{equation}\label{eq:qblock}
  \cQ_m(a,b)
  =\{(i,j)\in\mathbb Z^2:
      a\le i\le a+m-1,\ b\le j\le b+m-1\},
  \qquad a,b\in\mathbb Z.
\end{equation}
A refined region $\cR$ is \emph{generated by translated
$m\times m$ templates} if it is a finite union of such sets.
\end{definition}

The templates may overlap, and their union need not be rectangular.
Only cells available for subdivision at the current level are
included in $\cR$.

\begin{definition}[Channel length]\label{def:channel}
Let $\cR$ be a nonempty set of selected cell indices and let
$\boldsymbol q\in\cR$.
A translated template $\cQ_m(a,b)$ is called an \emph{$m$-channel
through $\boldsymbol q$ in $\cR$} if
\[
  \boldsymbol q\in\cQ_m(a,b)\subseteq\cR.
\]
Define
\[
  \operatorname{chl}_{\cR}(\boldsymbol q)
  =\max\{m:\boldsymbol q\in\cQ_m(a,b)\subseteq\cR
            \text{ for some }a,b\}.
\]
The \emph{channel length} of $\cR$ is
\begin{equation}\label{eq:channel-length}
  \chl(\cR)=\min_{\boldsymbol q\in\cR}\operatorname{chl}_{\cR}(\boldsymbol q).
\end{equation}
\end{definition}

\begin{example}\label{ex:template}\label{ex:channel}
Consider
\[
  \cR=\cQ_3(0,0)\cup\cQ_3(2,2),
\]
as shown in Figure~\ref{fig:channel}. The two templates overlap
in one cell, and every cell belongs to a $3\times3$ template
contained in $\cR$. The lower-left cell has index
$\boldsymbol q_*=(0,0)$ and belongs to no contained $4\times4$ template. Hence $\chl(\cR)=3$.
\end{example}

\begin{figure}[htbp]
\centering
\begin{tikzpicture}[scale=.74,line cap=round,line join=round]
  \fill[gray!22] (0,0) rectangle (3,3);
  \fill[gray!22] (2,2) rectangle (5,5);
  \fill[gray!55] (0,0) rectangle (1,1);
  \foreach \x in {0,...,5}{
    \draw[line width=.45pt] (\x,0)--(\x,5);
  }
  \foreach \y in {0,...,5}{
    \draw[line width=.45pt] (0,\y)--(5,\y);
  }
  \draw[line width=1pt,dashed] (0,0) rectangle (3,3);
  \draw[line width=1pt,densely dotted] (2,2) rectangle (5,5);
  \node[fill=white,inner sep=1pt] at (1.5,2.65) {$\cQ_3(0,0)$};
  \node[fill=white,inner sep=1pt] at (3.5,4.65) {$\cQ_3(2,2)$};
  \node at (.5,.5) {$\boldsymbol q_*$};
\end{tikzpicture}
\caption{A refined region of channel length three.}
\label{fig:channel}
\end{figure}
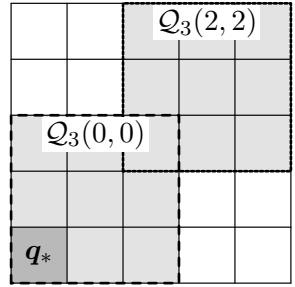

\begin{proposition}\label{prop:channeltemplate}
For a nonempty refined region $\cR$ and an integer $m\ge1$,
$\chl(\cR)\ge m$ if and only if $\cR$ is generated by translated
$m\times m$ templates. Consequently,
\begin{equation}\label{eq:channelmax}
  \chl(\cR)
  =\max\{m:\cR\text{ is generated by translated }
                    m\times m\text{ templates}\}.
\end{equation}
\end{proposition}

The proof is given in Appendix~\ref{app:templates}.
This equivalence expresses the same requirement in two ways:
every selected cell must belong to a contained template of the
prescribed size. It permits the refinement condition to be
stated either through template translations or through channel
length.

For the exclusion of vanishable T $l$-edges, define
\begin{equation}\label{eq:qd}
  q_d=\left\lceil\frac{d+1}{2}\right\rceil.
\end{equation}

\begin{proposition}\label{prop:minimaltemplate}
If a refined region $\cR$ satisfies $\chl(\cR)\ge q_d$, then
cross subdivision of its cells introduces no vanishable T
$l$-edges for $S_d$. Among conditions requiring a union of
translated $m\times m$ templates, $q_d$ is the smallest side
that guarantees this property for every refined region.
\end{proposition}

The threshold follows from the lower bound $2m+1$ for the number
of vertices on a new T $l$-edge crossing $m$ consecutive selected
cells. In particular, an isolated $m\times m$ submesh gives
vanishable T $l$-edges when $m<q_d$. The detailed argument is
given in Appendix~\ref{app:templates}.

The \emph{minimal no-vanishable template condition} requires
$\chl(\cR_k)\ge q_d$ at every refinement level. Here,
minimality refers only to a uniform square-template condition;
it is not a characterization of all no-vanishable refinements.

For the examples in Section~\ref{sec:examples}, the refinement
regions can now be written as
\begin{align*}
  \cR^{(5)}_1&=\cQ_3(1,1)\cup\cQ_3(3,3),&
  \cR^{(5)}_2&=\cQ_3(4,4)\cup\cQ_3(6,6),\\
  \cR^{(6)}_1&=\cQ_4(1,1)\cup\cQ_4(4,4),&
  \cR^{(6)}_2&=\cQ_4(6,6)\cup\cQ_4(9,9).
\end{align*}
Their channel lengths satisfy
\begin{equation}\label{eq:qequalsdminus2}
  \begin{aligned}
    \chl(\cR^{(5)}_1)=\chl(\cR^{(5)}_2)&=3=5-2,\\
    \chl(\cR^{(6)}_1)=\chl(\cR^{(6)}_2)&=4=6-2.
  \end{aligned}
\end{equation}
Thus, for $d=5,6$, even refinement by translated
$(d-2)\times(d-2)$ templates does not ensure dimensional
stability. This is stronger than the failure of unrestricted
refinement established in Section~\ref{sec:vanishable-examples}.

In contrast, the dimension formula in
\cite[Theorem~4.2]{HuangChen2026} gives stability under
$(d-1)\times(d-1)$ tensor-product subdivisions for $d\ge3$.
We therefore adopt the following assumption.

\begin{assumption}[Translated $(d-1)\times(d-1)$ refinement]
\label{ass:stabletemplate}
Let $d\ge2$. At every refinement level $k$, the nonempty region
$\cR_k$ selected for cross subdivision is a union of translated
$(d-1)\times(d-1)$ templates. Equivalently,
\begin{equation}\label{eq:stable-channel}
  \chl(\cR_k)\ge d-1.
\end{equation}
\end{assumption}

\begin{proposition}[\cite{DengChenJin2013,ZengEtAl2015,HuangChen2026}]
\label{prop:stable-template}
If a hierarchical T-mesh $\T$ satisfies
Assumption~\ref{ass:stabletemplate}, then the dimension of
$S_d(\T)$ is stable on its admissible structurally isomorphic
class.
\end{proposition}

\begin{corollary}\label{cor:dminus2}
For $d=5$ and $d=6$, replacing $(d-1)\times(d-1)$ templates in
Assumption~\ref{ass:stabletemplate} by
$(d-2)\times(d-2)$ templates does not ensure dimensional
stability, even if every refinement level is no-vanishable
and increases the dimension.
\end{corollary}

The justification of Proposition~\ref{prop:stable-template} and
the proof of Corollary~\ref{cor:dminus2} are given in
Appendix~\ref{app:templates}.

For $d\ge3$, the adopted condition also excludes vanishable
T $l$-edges because $d-1\ge q_d$. When $d=2$, the template
has size $1\times1$, so the assumption allows arbitrary levelwise
cross subdivision. Stability in this case follows from the known
biquadratic dimension formula \cite{DengChenJin2013,ZengEtAl2015},
and does not require the exclusion of vanishable T $l$-edges.

The four examples and the known stability result justify
Assumption~\ref{ass:stabletemplate} as a reasonable mesh
assumption for the $\PHtwoT$-spline series.
It ensures the dimensional stability required by
Section~\ref{sec:construction-suitable}; the subsequent papers
will construct the basis and establish its properties under
this assumption. We do not claim that this condition is
necessary for every stable hierarchical T-mesh, or that
$d-1$ is the smallest admissible template side in every degree.

\section{Conclusion}\label{sec:conclusion}

In this paper, we have studied a mesh assumption for the construction
of $\PHtwoT$-splines of bi-degree $(d,d)$ with the
highest order of smoothness over hierarchical T-meshes. We introduced
construction-suitable T-meshes and proved that dimensional stability
is necessary for a basis construction compatible with the mesh
structure. Four examples were then considered. The bicubic and
biquartic examples show that unrestricted hierarchical refinement may
introduce edges without increasing the dimension, and that the
dimension need not be stable. The examples of degrees five and six
show that excluding vanishable T $l$-edges is still insufficient,
even when every refinement level increases the dimension.

By introducing template translations and channel length, we expressed
the refinement conditions in terms of the arrangement of the selected
cells. For $d=5,6$, the examples satisfy the
$(d-2)\times(d-2)$ template condition but have unstable dimensions.
Together with the known stability result in \cite{HuangChen2026},
these examples justify the $(d-1)\times(d-1)$ template condition as
a reasonable mesh assumption for the subsequent basis construction.
The condition is sufficient for dimensional stability; it is not
claimed to be necessary for every stable hierarchical T-mesh or
optimal in every degree. The construction and properties of the
$\PHtwoT$-spline basis will be studied in the
subsequent parts of this series.

\appendix

\section{Conformality matrices and dimension calculations}
\label{app:conformality}\label{sec:conformality}

This appendix describes the reductions used to calculate the
dimensions in Section~\ref{sec:instability-assumption}. We use the
smoothing cofactor method and the decomposition of T-connected
components in \cite{LiDeng2016,HuangChen2024}. Row and column
indices in the explicit matrix certificates start at zero.

\subsection{The conformality matrix}

Let $n_{\mathrm T}$ be the number of T $l$-edges of $\T$, and let
$n_{\mathrm L}$ be the number of distinct vertices on $L(\T)$.
Collecting the equations
\eqref{eq:moment} for all T $l$-edges gives
\[
  M(\T)\boldsymbol\delta=0,
  \qquad M(\T)\in\R^{n_{\mathrm T}(d+1)\times n_{\mathrm L}}.
\]
Each vertex of $L(\T)$ corresponds to one column, including a vertex
shared by two T $l$-edges. For $w=(s_w,t_w)$, the entries are
\begin{equation}\label{eq:matrix-entries}
  M_{(\ell,k),w}=
  \begin{cases}
    s_w^k,&w\in\ell\text{ and }\ell\text{ is horizontal},\\
    t_w^k,&w\in\ell\text{ and }\ell\text{ is vertical},\\
    0,&w\notin\ell,
  \end{cases}
  \qquad k=0,\ldots,d.
\end{equation}
The kernel of $M(\T)$ is the conformality vector space associated
with $L(\T)$.

For the regular T-meshes considered in the four examples, there
are no rays and no homogeneous boundary conditions. The dimension
formula is \cite{LiDeng2016,HuangChen2024}
\begin{equation}\label{eq:dimformula}
  \dim S_d(\T)=(d+1)^2+c(d+1)+n_v-\rank M(\T),
\end{equation}
where $c$ is the number of cross-cuts and $n_v$ is the number of
all interior vertices. In general, $n_v\ne n_{\mathrm L}$. Only vertices on
$L(\T)$ are included in the columns of $M(\T)$; the remaining
$n_v-n_{\mathrm L}$ interior vertex cofactors are unconstrained
by this matrix.

\subsection{Vanishable edges and the CNDC}
\label{sec:decomposition-tools}

If an edge equation contains $r\le d+1$ remaining cofactor variables,
its Vandermonde matrix has full column rank. All these variables
are therefore zero. They are removed from every incident edge
equation, and the procedure is repeated until no further variables
are forced to zero. Each removed variable contributes one pivot
to the rank, not one free variable.

We recall the decomposition used when nonzero cofactor variables
remain. An order $\ell_1\succ\cdots\succ\ell_{n_{\mathrm T}}$ is
\emph{reasonable} if each $\ell_i$ has at least $d+1$ vertices after
its intersections with $\ell_1,\ldots,\ell_{i-1}$ have been removed.
A component admitting such an order is \emph{diagonalizable}
\cite{LiDeng2016}. Its conformality matrix has full row rank.

For a remaining component $L$, let $m(\ell\mid L)$ count the
vertices on $\ell$ that belong to no other edge of $L$. If
$m(\ell\mid L)\ge d+1$, the equations of $\ell$ can be eliminated
using $d+1$ of these mono-vertices. Remove $\ell$, retain its
intersection vertices on the other edges, and repeat. The
extracted edges form a diagonalizable component in reverse
extraction order, with the shared vertices assigned to the
remaining component. If the process
terminates with a nonempty component $L_{\mathrm C}$, then every
edge of $L_{\mathrm C}$ has fewer than $d+1$ mono-vertices. This
is the completely non-diagonalizable component (CNDC) of the
complete partition in \cite{HuangChen2024}.

An extracted edge with $m$ relative mono-vertices contributes
$d+1$ pivots and $m-d-1$ free variables. Consequently, extracting
$e$ edges gives
\begin{equation}\label{eq:rankpartition}
  \rank M=e(d+1)+\rank M(L_{\mathrm C}),
\end{equation}
when no prior vanishable-edge reduction is needed. Pivots from
any earlier forced-zero reduction are added separately.

\subsection{Elimination of mono-vertices}

Let $A_\ell=\{a_1,\ldots,a_{\mu_\ell}\}$ be the tangential
coordinates of the mono-vertices on a remaining edge $\ell$,
where $\mu_\ell\le d$. Set
\[
  p_\ell(z)=\prod_{a\in A_\ell}(z-a),
\]
with the empty product equal to one. Here $z=s$ for a horizontal edge and $z=t$ for a vertical edge.
Eliminating the mono-vertex variables leaves the matrix
\begin{equation}\label{eq:condensedentry}
  K_{(\ell,k),w}=
  \begin{cases}
    z_w^k p_\ell(z_w),&w\in\ell,\\
    0,&w\notin\ell,
  \end{cases}
  \qquad k=0,\ldots,d-\mu_\ell,
\end{equation}
whose columns correspond to the remaining multi-vertices.

To justify the reduction, complete
$p_\ell(z),zp_\ell(z),\ldots,z^{d-\mu_\ell}p_\ell(z)$
by the $\mu_\ell$ Lagrange polynomials at the mono-vertices.
These polynomials form a basis of $\mathbb P_d$. The associated
change of row basis produces an identity block on the mono-columns,
while the other rows vanish on those columns. Column operations
then eliminate the entries in the pivot rows outside this identity
block. Thus each mono-vertex contributes one pivot. Mono-columns
of different edges are disjoint, so the reductions can be performed
independently.

For the examples without vanishable edges, let
$\mu=\sum_{\ell\in L_{\mathrm C}}\mu_\ell$ and
$\rho=e(d+1)+\mu$. The full reduction has the form
\begin{equation}\label{eq:blockreduction}
  M\sim
  \begin{pmatrix}
    I_\rho&0&0\\
    0&K&0
  \end{pmatrix},
  \qquad \rank M=\rho+\rank K.
\end{equation}
The last block of columns records the free variables of the
extracted component. These columns must be retained when specifying
the size of the full matrix.

The first refinement levels of the degree-five and degree-six
examples have a particularly simple reduced system. On a complete
$r\times r$ array of crossings, suppose that the equations are
\[
  \sum_{i=1}^r \omega_i\delta_{ij}=0\quad(1\le j\le r),
  \qquad
  \sum_{j=1}^r \vartheta_j\delta_{ij}=0\quad(1\le i\le r),
\]
with all $\omega_i,\vartheta_j$ nonzero. Setting
$\zeta_{ij}=\omega_i \vartheta_j\delta_{ij}$ reduces these equations to zero
row and column sums. The total row sum equals the total column sum. Conversely,
the first $r-1$ row sums and all $r$ column sums are independent,
so the rank is $2r-1$. This calculation is
independent of the nonzero weights.

\subsection{Exact rank certificates}
\label{app:rank-certificates}

To establish the rank of a rational matrix, both an upper and a
lower bound are needed. Let $\mathcal I,\mathcal J$ select a nonsingular submatrix
$A=K[\mathcal I,\mathcal J]$ of order $r$, and put
\[
  B=K[\mathcal I,\mathcal J^c],\qquad C=K[\mathcal I^c,\mathcal J],\qquad D=K[\mathcal I^c,\mathcal J^c].
\]
All index sets are ordered increasingly, and complements are taken
in the full row or column index set of $K$. Block elimination gives
\begin{equation}\label{eq:schur-rank}
  \rank K=r+\rank(D-CA^{-1}B).
\end{equation}
Thus $\det A\ne0$ together with
\begin{equation}\label{eq:schur-zero}
  D-CA^{-1}B=0
\end{equation}
proves $\rank K=r$. The certificates below use exact rational
arithmetic. No numerical rank tolerance is involved.

For compactness, write
\[
  H(q;a,b)=\{(x,q):a\le x\le b\},\qquad
  V(q;a,b)=\{(q,y):a\le y\le b\}.
\]
Unless otherwise stated, horizontal edges precede vertical edges
and are ordered by ordinate; vertical edges are ordered by
abscissa. Rows within an edge block follow increasing $k$ in
\eqref{eq:condensedentry}, and multi-vertices are ordered
lexicographically by their coordinates.

\section{The bicubic example}\label{app:d3}

We prove Theorem~\ref{thm:d3} for the meshes defined by
\eqref{eq:d3knots}--\eqref{eq:d3R3}. The final mesh has nineteen
T $l$-edges, eight cross-cuts and eighty-six interior vertices.
Seventy vertices lie on the T $l$-edges. Hence
\[
  M_3(\tau)\in\R^{76\times70},\qquad
  \dim S_3(\T_3(\tau))=134-\rank M_3(\tau).
\]

At $\tau=1$, successively eliminate the edges on
\[
  y=\frac32,\frac{17}8,\frac{23}8,\frac{25}8,
    \frac{13}4,\frac{27}8,
  \qquad x=\frac32,\frac{13}4,
\]
in the displayed order. The numbers of newly forced-zero variables
are $3,3,3,3,4,3,3,3$, respectively, totaling twenty-five.
The corresponding edges give the same elimination for every
$\tau>0$. The eleven remaining edges have twenty mono-vertices
and twenty-five multi-vertices. Therefore,
\begin{equation}\label{eq:d3reduction}
  M_3(\tau)\sim
  \begin{pmatrix}
    I_{45}&0\\
    0&K_3(\tau)\\
    0_{7\times45}&0_{7\times25}
  \end{pmatrix},
  \qquad K_3(\tau)\in\R^{24\times25}.
\end{equation}
In particular,
\begin{equation}\label{eq:d3nullity}
  \dim S_3(\T_3(\tau))=64+\dim\ker K_3(\tau).
\end{equation}

Use the affine coordinates
\begin{equation}\label{eq:d3affine}
  \xi=8x-8(\tau-1),\qquad \eta=8y.
\end{equation}
An invertible affine change of tangential coordinate does not
change the rank of an edge equation block. In these coordinates,
all remaining multi-vertices are independent of $\tau$.
Table~\ref{tab:d3blocks} gives the row blocks of $K_3(\tau)$,
and Table~\ref{tab:d3vectors} specifies the column order.
Here $H_q$ lies on $\eta=q$ and $V_q$ lies on $\xi=q$.
Substitution into \eqref{eq:condensedentry} completely specifies
the matrix.

\begin{table}[htbp]
\centering\small
\caption{Edge blocks of $K_3(\tau)$.}\label{tab:d3blocks}
\begin{tabular}{@{}cllc@{}}
\toprule
Edge & $A_\ell$ & Multi-vertex coordinates & Rows\\
\midrule
$H_{18}$&$\{24\}$&$20,21,22,23,28$&$0$--$2$\\
$H_{19}$&$\{24\}$&$20,21,22,23$&$3$--$5$\\
$H_{20}$&$\{16,24,32\}$&$20,21,22,23,28$&$6$\\
$H_{21}$&$\{24\}$&$20,21,22,23$&$7$--$9$\\
$H_{22}$&$\{24\}$&$20,21,22,23$&$10$--$12$\\
$H_{28}$&$\{16-8\tau,16,24\}$&$20,22,23$&$13$\\
$V_{20}$&$\{16,24,32\}$&$18,19,20,21,22,28$&$14$\\
$V_{21}$&$\{16\}$&$18,19,20,21,22$&$15$--$17$\\
$V_{22}$&$\{16,24\}$&$18,19,20,21,22,28$&$18$--$19$\\
$V_{23}$&$\{24\}$&$18,19,20,21,22,28$&$20$--$22$\\
$V_{28}$&$\{8,16,24\}$&$18,20$&$23$\\
\bottomrule
\end{tabular}
\end{table}

\begin{table}[htbp]
\centering\small
\caption{Multi-vertices and kernel vectors for $K_3$.}
\label{tab:d3vectors}
\setlength{\tabcolsep}{11pt}
\begin{tabular}{@{}rrrrr@{}}
\toprule
$j$&$\xi_j$&$\eta_j$&$a_j$&$b_j$\\
\midrule
0&20&18&-1&-440\\
1&20&19&4&800\\
2&20&20&-6&-565\\
3&20&21&4&160\\
4&20&22&-1&0\\
5&20&28&0&0\\
6&21&18&4&1600\\
7&21&19&-16&-3200\\
8&21&20&24&2400\\
9&21&21&-16&-640\\
10&21&22&4&0\\
11&22&18&-6&-2080\\
12&22&19&24&4800\\
13&22&20&-36&-3900\\
14&22&21&24&960\\
15&22&22&-6&0\\
16&22&28&0&-20\\
17&23&18&4&1088\\
18&23&19&-16&-3200\\
19&23&20&24&2880\\
20&23&21&-16&-640\\
21&23&22&4&0\\
22&23&28&0&32\\
23&28&18&0&-8\\
24&28&20&0&5\\
\bottomrule
\end{tabular}
\end{table}

Only row $13$ depends on $\tau$. Its nonzero entries, at the
columns corresponding to $(20,28),(22,28),(23,28)$, are
\begin{equation}\label{eq:d3parameterrow}
  -64-128\tau,\qquad -72-96\tau,\qquad -49-56\tau.
\end{equation}
Delete this row to obtain the constant matrix
$C_3\in\mathbb Z^{23\times25}$. For
$\mathcal J_C=\{0,1,\ldots,20,22,23\}$, exact elimination gives
\begin{equation}\label{eq:d3minor}
  \det C_3[:,\mathcal J_C]=-2^{58}3^{10}\cdot5\ne0.
\end{equation}
Thus $\rank C_3=23$. The vectors $\boldsymbol a,\boldsymbol b$ in
Table~\ref{tab:d3vectors} satisfy
\begin{equation}\label{eq:d3vectors}
  C_3\boldsymbol a=C_3\boldsymbol b=0,\qquad K_3(\tau)\boldsymbol a=0,\qquad
  K_3(\tau)\boldsymbol b=128(\tau-1)e_{13},
\end{equation}
where $e_{13}$ is the coordinate vector for row $13$.
These identities follow directly from the two tables. For example,
the only potentially nonzero entry of $K_3(\tau)\boldsymbol b$ is
\[
  (-72-96\tau)(-20)+(-49-56\tau)32=128(\tau-1).
\]
Since $\boldsymbol a\ne0$, $a_{24}=0$ and $b_{24}=5$, the two vectors are
independent and span $\ker C_3$. Therefore
\[
  \rank K_3(\tau)=
  \begin{cases}
    23,&\tau=1,\\
    24,&\tau>0,\ \tau\ne1.
  \end{cases}
\]
Together with \eqref{eq:d3nullity}, this proves
Theorem~\ref{thm:d3}. As an additional check, for
$\mathcal J_K=\mathcal J_C\cup\{24\}$ in increasing order,
\begin{equation}\label{eq:d3fullminor}
  \det K_3(\tau)[:,\mathcal J_K]=-2^{65}3^{10}(\tau-1).
\end{equation}

The stage counts underlying Table~\ref{tab:d3stages} are listed
below. At levels one and two, repeated forced-zero elimination
removes all cofactor variables, giving full column ranks $21$
and $38$. At level three, \eqref{eq:d3reduction} gives ranks
$68$ and $69$.

\begin{table}[htbp]
\centering\small
\caption{Conformality counts for the bicubic example.}
\label{tab:d3-rank-counts}
\begin{tabular}{@{}crrrrrc@{}}
\toprule
Level&$n_{\mathrm T}$&$c$&$n_v$&$n_{\mathrm L}$&Rows of $M$&
$\rank M$ for $\tau=1,2$\\
\midrule
0&0&8&16&0&0&$0,0$\\
1&6&8&37&21&24&$21,21$\\
2&11&8&54&38&44&$38,38$\\
3&19&8&86&70&76&$68,69$\\
\bottomrule
\end{tabular}
\end{table}

\section{The biquartic example}\label{app:d4}

We prove Theorem~\ref{thm:d4}. The mesh defined by
\eqref{eq:d4knots} and \eqref{eq:d4mask} has eight T $l$-edges,
ten cross-cuts and fifty-seven interior vertices. There are
thirty-two vertices on the T $l$-edges, so
\[
  M_4(\tau)\in\R^{40\times32},\qquad
  \dim S_4(\T_4(\tau))=132-\rank M_4(\tau).
\]
The four dashed T $l$-edges in Figure~\ref{fig:d4mesh} each have
three vertices. Their twelve cofactor variables are zero.
The remaining four edges are
\[
  H(5/2;2,5),\quad H(7/2;\tau,4),\quad
  V(5/2;2,5),\quad V(7/2;1,4).
\]
After the zero variables are removed, their mono-vertex coordinate
sets, in the same order, are
\[
  \{2,3,4,5\},\quad\{\tau,2,3,4\},\quad
  \{2,3,4,5\},\quad\{1,2,3,4\}.
\]
The sixteen mono-vertices can be eliminated. Order the four
remaining multi-vertices as
\[
  (5/2,5/2),\quad(5/2,7/2),\quad
  (7/2,5/2),\quad(7/2,7/2).
\]
Using \eqref{eq:condensedentry} and multiplying each resulting
row by $16/3$ gives
\begin{equation}\label{eq:d4K}
  K_4(\tau)=
  \begin{pmatrix}
    -5&0&3&0\\
    0&5-2\tau&0&-(7-2\tau)\\
    -5&3&0&0\\
    0&0&3&-5
  \end{pmatrix}.
\end{equation}
Thus
\begin{equation}\label{eq:d4reduction}
  M_4(\tau)\sim
  \begin{pmatrix}
    I_{28}&0\\
    0&K_4(\tau)\\
    0_{8\times28}&0_{8\times4}
  \end{pmatrix}.
\end{equation}
The determinant and a fixed minor satisfy
\[
  \det K_4(\tau)=60(\tau-1),\qquad
  \det K_4(\tau)[\{0,2,3\},\{0,1,2\}]=-45.
\]
Hence $\rank K_4(1)=3$ and $\rank K_4(\tau)=4$ for
$0<\tau<2$, $\tau\ne1$. At $\tau=1$, the vector
$(3,5,5,3)^{\mathsf T}$ spans the kernel. Consequently,
\[
  \dim S_4(\T_4(\tau))=104-\rank K_4(\tau),
\]
which proves Theorem~\ref{thm:d4}, including the two meshes
$\tau=1$ and $\tau=3/2$ displayed in the main text.

\section{The degree-five example}
\label{app:d5}\label{app:d5matrix}\label{app:reduction}

We calculate the dimensions in \eqref{eq:d5sequence} for the
subdivisions \eqref{eq:d5R1}--\eqref{eq:d5R2}. Put
$x_1=1+u/10$, where $u=0$ and $u=1$ give $\T_{5,A}$ and
$\T_{5,B}$, respectively. A vertex $(s,t)$ in the uniform mesh
corresponds to $(g_u(s),t)$ in the other mesh, where
\begin{equation}\label{eq:perturbation-map}
  g_u(s)=
  \begin{cases}
    (1+u/10)s,&0\le s\le1,\\
    s+(u/10)(2-s),&1\le s\le2,\\
    s,&s\ge2.
  \end{cases}
\end{equation}
This map specifies the vertex coordinates obtained from the
prescribed midpoint subdivisions. We use the same coordinate
convention for the degree-six example.

\subsection{The first refinement level}

The initial dimension is $(7+5)^2=144$. At level one, the
conformality matrix $M_5^1(u)$ has size $60\times61$. Extract
$H(7/2;1,6)$ and then $V(7/2;1,6)$, using their corresponding
positions when $u=1$. Their relative mono-vertex counts are six
and seven. They contribute twelve pivots and one free variable.
The eight remaining edges have forty mono-vertices and eight
multi-vertices. After eliminating the mono-vertices, the crossings
form two disjoint $2\times2$ arrays. On each array, the horizontal
edges have the same mono-coordinate set, and so do the vertical
edges. The weighted row-and-column calculation in
Appendix~\ref{app:conformality}
therefore gives rank three for each array. Thus
\[
  \rank M_5^1(u)=12+40+3+3=58.
\]
Since $c=12$ and $n_v=97$, \eqref{eq:dimformula} gives
$\dim S_5(\T_{5,A}^1)=\dim S_5(\T_{5,B}^1)=147$.
Here the superscript denotes the refinement level.

\subsection{The final mesh and its reduced matrix}

The final mesh has twenty T $l$-edges and $122$ vertices on their
union. Extract the same central horizontal and vertical edges.
Their relative mono-vertex counts are again six and seven.
The remaining eighteen edges have at most five mono-vertices
each and form the CNDC. They contain seventy-two mono-vertices
and thirty-seven multi-vertices. Consequently,
\begin{equation}\label{eq:d5red}
  M_5(u)\sim
  \begin{pmatrix}
    I_{84}&0&0_{84\times1}\\
    0&K_5(u)&0_{36\times1}
  \end{pmatrix},
  \qquad K_5(u)\in\Q^{36\times37}.
\end{equation}

Table~\ref{tab:d5blocks} specifies the reduced matrix in the
uniform geometry. Each row gives a horizontal edge $H(q;a,b)$
and a vertical edge $V(q;a,b)$. Their tangential mono-coordinate
sets are both $A$. The multi-vertices are precisely the
intersections of the listed horizontal and vertical segments.
For $u=1$, transform all $x$-coordinates by
\eqref{eq:perturbation-map}; the $y$-coordinates do not change.
This prescription, the row ranges in the table and the
lexicographic column order specify every entry through
\eqref{eq:condensedentry}.

\begin{table}[htbp]
\centering\small
\caption{Edge data for $K_5(0)$.}\label{tab:d5blocks}
\begin{tabular}{@{}cclcc@{}}
\toprule
$q$&$[a,b]$&$A$&Horizontal rows&Vertical rows\\
\midrule
$3/2$&$[1,4]$&$\{1,2,3,7/2,4\}$&$0$&$18$\\
$9/4$&$[2,7/2]$&$\{2,3,7/2\}$&$1$--$3$&$19$--$21$\\
$5/2$&$[1,4]$&$\{1,2,3,7/2,4\}$&$4$&$22$\\
$11/4$&$[2,7/2]$&$\{2,3,7/2\}$&$5$--$7$&$23$--$25$\\
$13/4$&$[2,9/2]$&$\{2,3,7/2,4\}$&$8$--$9$&$26$--$27$\\
$15/4$&$[3,9/2]$&$\{3,7/2,4\}$&$10$--$12$&$28$--$30$\\
$17/4$&$[3,9/2]$&$\{3,7/2,4\}$&$13$--$15$&$31$--$33$\\
$9/2$&$[3,6]$&$\{3,7/2,4,5,6\}$&$16$&$34$\\
$11/2$&$[3,6]$&$\{3,7/2,4,5,6\}$&$17$&$35$\\
\bottomrule
\end{tabular}
\end{table}

\subsection{Rank certificates and dimensions}

For $u=0$, choose $\mathcal I_0,\mathcal J_0$ through their complements
\[
  \mathcal I_0^c=\{25,32,33,35\},\qquad
  \mathcal J_0^c=\{8,18,32,33,36\}.
\]
For $u=1$, choose
\[
  \mathcal I_1^c=\{25,33,35\},\qquad
  \mathcal J_1^c=\{8,18,33,36\}.
\]
The row and column universes are $\{0,\ldots,35\}$ and
$\{0,\ldots,36\}$. Exact evaluation gives
\begin{align}
  \det K_5(0)[\mathcal I_0,\mathcal J_0]
    &=-\frac{3^{39}5^{15}\cdot7\cdot11}{2^{189}},
      \label{eq:d5minor0}\\
  \det K_5(1)[\mathcal I_1,\mathcal J_1]
    &=-\frac{3^{39}5^{10}7^2\cdot13\cdot29}{2^{191}}.
      \label{eq:d5minor1}
\end{align}
The corresponding Schur complements in \eqref{eq:schur-zero}
are exactly $0_{4\times5}$ and $0_{3\times4}$. Hence
\begin{equation}\label{eq:d5Kranks}
  \rank K_5(0)=32,\qquad \rank K_5(1)=33.
\end{equation}
Adding the eighty-four pivots in \eqref{eq:d5red} gives full
ranks $116$ and $117$. The final mesh has $c=12$ and $n_v=158$,
so
\[
  \dim S_5(\T_{5,A})=266-116=150,\qquad
  \dim S_5(\T_{5,B})=266-117=149.
\]
Together with the first-level calculation, this proves
\eqref{eq:d5sequence}. The extra zero column in \eqref{eq:d5red}
accounts for the free variable of the extracted component.

\section{The degree-six example}
\label{app:d6}\label{app:d6matrix}

We use the subdivisions \eqref{eq:d6R1}--\eqref{eq:d6R2} and
again put $x_1=1+u/10$, with $u=0,1$. The correspondence between
vertex coordinates is given by \eqref{eq:perturbation-map}.

\subsection{The first refinement level}

The initial dimension is $(10+6)^2=256$. At level one, extract
$H(9/2;1,8)$ and then $V(9/2;1,8)$. Their relative mono-vertex
counts are eight and nine, contributing fourteen pivots and
three free variables. The twelve remaining edges have seventy-two
mono-vertices and eighteen multi-vertices. After mono-vertex
elimination, the crossings form two disjoint $3\times3$ arrays.
As in the degree-five calculation, each array has rank five.
Thus the level-one conformality matrix $M_6^1(u)$, of size
$98\times107$, satisfies
\[
  \rank M_6^1(u)=14+72+5+5=96
\]
for both geometries. Since $c=18$ and $n_v=188$, the dimension
at this level is $49+18\cdot7+188-96=267$.

\subsection{The final mesh and its reduced matrix}

The final mesh has twenty-eight T $l$-edges and $214$ vertices
on their union. Extract the same central horizontal and vertical
edges. Their relative mono-vertex counts remain eight and nine.
The other twenty-six edges have at most six relative mono-vertices
each and form the CNDC. It contains $124$ mono-vertices and
seventy-three multi-vertices. Therefore,
\begin{equation}\label{eq:d6red}
  M_6(u)\sim
  \begin{pmatrix}
    I_{138}&0&0_{138\times3}\\
    0&K_6(u)&0_{58\times3}
  \end{pmatrix},
  \qquad K_6(u)\in\Q^{58\times73}.
\end{equation}
Table~\ref{tab:d6blocks} specifies its edge blocks using the same
paired horizontal--vertical convention as Table~\ref{tab:d5blocks}.
Together with \eqref{eq:perturbation-map} and
\eqref{eq:condensedentry}, it determines both matrices exactly.

\begin{table}[htbp]
\centering\small
\caption{Edge data for $K_6(0)$.}\label{tab:d6blocks}
\begin{tabular}{@{}cclcc@{}}
\toprule
$q$&$[a,b]$&$A$&Horizontal rows&Vertical rows\\
\midrule
$3/2$&$[1,5]$&$\{1,2,3,4,9/2,5\}$&$0$&$29$\\
$5/2$&$[1,5]$&$\{1,2,3,4,9/2,5\}$&$1$&$30$\\
$13/4$&$[3,5]$&$\{3,4,9/2,5\}$&$2$--$4$&$31$--$33$\\
$7/2$&$[1,5]$&$\{1,2,3,4,9/2,5\}$&$5$&$34$\\
$15/4$&$[3,5]$&$\{3,4,9/2,5\}$&$6$--$8$&$35$--$37$\\
$17/4$&$[3,5]$&$\{3,4,9/2,5\}$&$9$--$11$&$38$--$40$\\
$19/4$&$[3,13/2]$&$\{3,4,9/2,5,6\}$&$12$--$13$&$41$--$42$\\
$21/4$&$[9/2,13/2]$&$\{9/2,5,6\}$&$14$--$17$&$43$--$46$\\
$11/2$&$[4,8]$&$\{4,9/2,5,6,7,8\}$&$18$&$47$\\
$23/4$&$[9/2,13/2]$&$\{9/2,5,6\}$&$19$--$22$&$48$--$51$\\
$25/4$&$[9/2,13/2]$&$\{9/2,5,6\}$&$23$--$26$&$52$--$55$\\
$13/2$&$[4,8]$&$\{4,9/2,5,6,7,8\}$&$27$&$56$\\
$15/2$&$[4,8]$&$\{4,9/2,5,6,7,8\}$&$28$&$57$\\
\bottomrule
\end{tabular}
\end{table}

\subsection{Rank certificates and dimensions}

Use the row universe $\{0,\ldots,57\}$ and the column universe
$\{0,\ldots,72\}$. For $u=0$, let
\begin{align*}
  \mathcal I_0^c&=\{40,55,57\},\\
  \mathcal J_0^c&=\{4,5,12,14,26,27,31,32,46,61,62,65,66,67,68,69,71,72\}.
\end{align*}
For $u=1$, let
\begin{align*}
  \mathcal I_1^c&=\{40,55\},\\
  \mathcal J_1^c&=\{4,5,12,14,26,27,31,32,61,62,65,66,67,68,69,71,72\}.
\end{align*}
The selected minors have orders fifty-five and fifty-six, with
\begin{align}
  \det K_6(0)[\mathcal I_0,\mathcal J_0]
    &=-\frac{3^{72}5^{38}7^{18}\cdot11\cdot13^2}{2^{377}},
      \label{eq:d6minor0}\\
  \det K_6(1)[\mathcal I_1,\mathcal J_1]
    &=-\frac{3^{80}5^{22}7^{18}\cdot11\cdot13^2\cdot19
                \cdot23^2\cdot29^2\cdot59^2}{2^{396}}.
      \label{eq:d6minor1}
\end{align}
Their Schur complements are exactly $0_{3\times18}$ and
$0_{2\times17}$, respectively. By \eqref{eq:schur-rank},
\begin{equation}\label{eq:d6Kranks}
  \rank K_6(0)=55,\qquad \rank K_6(1)=56.
\end{equation}
Thus the full ranks are $138+55=193$ and $138+56=194$.
The final mesh has $c=18$ and $n_v=295$, yielding
\[
  \dim S_6(\T_{6,A})=470-193=277,\qquad
  \dim S_6(\T_{6,B})=470-194=276.
\]
This proves \eqref{eq:d6sequence}. The three zero columns in
\eqref{eq:d6red} retain the free variables of the extracted
component. Together with Appendix~\ref{app:d5}, the calculation
also proves Theorem~\ref{thm:main}.

\section{The template conditions}\label{app:templates}

\subsection{Channel length and template translations}

\begin{proof}[Proof of Proposition~\ref{prop:channeltemplate}]
Suppose $\chl(\cR)\ge m$. Every selected cell index $\boldsymbol q\in\cR$ belongs to a
square template of side at least $m$ contained in $\cR$.
Within that template, choose an $m\times m$ subtemplate containing
$\boldsymbol q$. The union of these subtemplates is $\cR$. Conversely,
if $\cR$ is a union of translated $m\times m$ templates,
every cell belongs to one of them, so $\chl(\cR)\ge m$.
Taking the largest such $m$ proves \eqref{eq:channelmax}.
\end{proof}

\subsection{Exclusion of vanishable edges}

\begin{proof}[Proof of Proposition~\ref{prop:minimaltemplate}]
Consider a maximal horizontal sequence of $r$ consecutive cells
selected for subdivision. Since every selected cell belongs to a
contained $q_d\times q_d$ template, the sequence has $r\ge q_d$.
Its new horizontal midline meets $r+1$ vertical grid lines and
$r$ newly inserted vertical midlines. Hence the corresponding
new T $l$-edge has at least $2r+1$ vertices. As
\[
  2r+1\ge2q_d+1\ge d+2,
\]
the edge is non-vanishable. The vertical case is identical.
Any continuation of the segment adds vertices and cannot invalidate
the bound.

For minimality, cross-subdivide an isolated interior
$m\times m$ submesh with $m<q_d$. Each new T $l$-edge has
exactly $2m+1\le d+1$ vertices and is vanishable. Thus no smaller
uniform square-template side gives the stated guarantee.
\end{proof}

\subsection{The adopted mesh assumption}

\begin{proof}[Justification of Proposition~\ref{prop:stable-template}]
First let $d\ge3$. By Proposition~\ref{prop:channeltemplate},
Assumption~\ref{ass:stabletemplate} means that each level is obtained
by cross-subdividing a collection of $(d-1)\times(d-1)$
tensor-product cell submeshes, with overlaps permitted.
Since $d-1\ge q_d$, these subdivisions introduce no vanishable
T $l$-edges. The dimension formula in
\cite[Theorem~4.2]{HuangChen2026} therefore applies. Its mesh counts
and numbers of isolated subdivision components are unchanged on the
admissible structurally isomorphic class with prescribed subdivisions.
Hence the dimension is stable.

When $d=2$, the $1\times1$ template condition imposes no additional
restriction on the selected cells. The biquadratic dimension formula
for hierarchical T-meshes in \cite{DengChenJin2013,ZengEtAl2015}
depends only on vertex and edge counts and the isolated subdivisions
at the successive levels. These quantities are also unchanged on
the specified class. This proves the assertion for $d=2$.
\end{proof}

\begin{proof}[Proof of Corollary~\ref{cor:dminus2}]
For $d=5$ and $d=6$, the regions in
\eqref{eq:d5R1}--\eqref{eq:d6R2} are generated by translated
$(d-2)\times(d-2)$ templates at both levels. They introduce no
vanishable T $l$-edges. Appendices~\ref{app:d5} and~\ref{app:d6}
show that each refinement level increases the dimension, but the
final dimensions differ on two structurally isomorphic meshes.
Thus the smaller template condition does not ensure stability.
\end{proof}

\bibliographystyle{unsrt}
\bibliography{PH2T_PartI}

\end{document}